\documentclass[12pt]{article}%
\usepackage{amsfonts}
\usepackage{mathrsfs}
\usepackage{fancyhdr}
\usepackage{comment}
\usepackage[a4paper, top=2.5cm, bottom=2.5cm, left=2.2cm, right=2.2cm]%
{geometry}
\usepackage{times}
\usepackage{amsmath}
\usepackage{changepage}
\usepackage{mathtools}
\usepackage{amssymb}
\usepackage{graphicx}%
\usepackage{amsthm}
\usepackage{tikz}
\usepackage{verbatim}
\usepackage{hyperref}
\usetikzlibrary{arrows.meta}
\usetikzlibrary{positioning}
\newtheorem{theorem}{Theorem}[section]

\newtheorem{lemma}[theorem]{Lemma}

\newtheorem{definition}[theorem]{Definition}
\newtheorem{example}[theorem]{Example}

\newtheorem{remark}[theorem]{Remark}
\newtheorem{assumption}{A}

\newcommand{\N}{\mathbb{N}}

\newcommand{\R}{\mathbb{R}}

\newcommand{\B}{{\partial\Omega}}
\newcommand{\n}{\vec{n}}

\newcommand{\bb}[1]{\mathbb{#1}}

\title{Ergodic Mean Field Games of Controls with State Constraints}
\author{Jameson Graber, Kyle Rosengartner}
\date{April 2026}

\begin{document}

\maketitle

\begin{abstract}
    In a mean field game of controls, players seek to minimize a cost that depends on the joint distribution of players' states and controls. We consider an ergodic problem for second-order mean field games of controls with state constraints, in which equilibria are characterized by solutions to a second-order MFGC system where the value function blows up at the boundary, the density of players vanishes at a commensurate rate, and the joint distribution of states and controls satisfies the appropriate fixed-point relation. We prove that such systems are well-posed in the case of monotone coupling and Hamiltonians with at most quadratic growth.
\end{abstract}

\section{Introduction}

A mean field game (MFG) is a type of differential game, usually consisting of a continuum of identical players, in which each player seeks to minimize a cost (or maximize a utility) that depends on the distribution of the players' states. The theory of mean field games was introduced independently by Lasry and Lions in \cite{lasry2007mean} and by Caines, Huang, and Malham\'e in \cite{huang2006large}. It is well known that a Nash equilibrium to such a game is characterized by a coupled system of PDE known as the MFG system, in which the value function satisfies a Hamilton-Jacobi equation while the distribution of player states satisfies a Fokker-Planck equation.

In applications, it is often natural to require that players remain within a particular domain, thereby forcing players to restrict their class of admissible controls to those which lead to players remaining within the domain or its closure (at least with probability $1$). This is called a \textit{state constraints problem}. In \cite{lasry1989nonlinear}, Lasry and Lions consider models of stochastic control problems with state constraints, as well as their associated nonlinear second-order elliptic PDE. In \cite{cannarsa2018existence,cannarsa2021mean}, the authors investigate the well-posedness of the MFG system in the dynamic (time-dependent), deterministic (first-order) case. \cite{achdou2022deterministic} also considers deterministic, dynamic mean field games with state constraints, but in the case where agents control their acceleration rather than their velocities. A recent paper (see \cite{cutri2026constrained}) studied a class of constrained mean field games with Grushin type dynamics.

In this paper we focus on the study of ergodic problems for mean field games with state constraints. The most closely related results are found in \cite{porretta2023ergodic,sardarli2021ergodic}, which study second order ergodic mean field games with coupling that depends only on the distribution of states (which is still the most common case in the MFG literature). 
The general approach of these papers, which we adopt in the case of MFGC, is to carefully combine the analysis of Hamilton-Jacobi equations with state constraints (which goes back to \cite{lasry1989nonlinear}) with new results on the Fokker-Planck equation whose solution vanishes at a rate commensurate with the blow-up of solutions to the Hamilton-Jacobi equation.
Without state constraints, there are many works in the literature on second order ergodic mean field games. See, for instance, \cite{bernardini2023ergodic,cacace2018ergodic,cao2023stationary,cardaliaguet2013long,dianetti2023ergodic,dragoni2018ergodic,kong2026mountain}.

Related to the idea of state constraints are the notions of \textit{reflecting boundary conditions} and \textit{invariance constraints on the state space}. In both cases, players are again forced to remain within the domain. However, instead of doing so by restricting the class of controls, this is done by either introducing a reflection term to the underlying state dynamics or by making assumptions on the relationship between the drift and diffusion terms. In \cite{lions1984stochastic}, many of the foundations for analyzing reflected SDEs were developed which would later be use in studying reflecting MFGs. Some important references on mean field games with reflecting boundary conditions include \cite{gomes2023time,ricciardi2022master,ricciardi2023convergence}. The case of invariance constraints has been studied in relation to MFGs (see \cite{porretta2020mean}), the Master equation (see \cite{zitridis2022master}), and mean field games of controls (see \cite{graber2025mean}).

In contrast to a traditional MFG, in a mean field game of controls (MFGC), each player's cost depends not only on the distribution of players' states but also on their controls. In MFGCs, a Nash equilibrium corresponds to a system of PDE similar to the MFG system. However, in MFGCs, the joint distribution $\mu$ of states and controls must satisfy an additional fixed-point relation, as the optimal feedback control corresponding to a given distribution $\mu$ must be compatible with $\mu$ itself. This type of game has elsewhere been referred to as an \textit{extended mean field game} (see \cite{gomes2014existence,gomes2016extended}), but the terminology ``mean field game of controls" now appears to be standard, cf.~\cite{cardaliaguet2018mean}. Compared to traditional MFGs, MFGCs have received far less attention in the literature.

Kobeissi's 2022 papers \cite{kobeissi2022mean,kobeissi2022classical} give a comprehensive analysis of the well-posedness of second order MFGCs on $\mathbb{T}^n$ and $\R^n$ under both monotone and non-monotone couplings. Later papers investigated the case of Dirichlet boundary conditions under the assumption that the set of admissible controls is bounded (see \cite{bongini2024mean}) and provided probabilistic results for mean field games of controls with reflecting boundary conditions (see \cite{bo2025mean}). In \cite{graber2025mean}, we extended Kobeissi's results to the cases of Dirichlet and Neumann boundary conditions as well as to the case of invariance constraints. Finally, \cite{graber2021note} investigates the existence of mild solutions to first-order mean field games of controls under state constraints.

The purpose of this article is to investigate the ergodic problem for second-order MFGCs with state constraints (see \eqref{Eq: Ergodic MFG}). We prove that this system is well-posed under relatively generic assumptions, and we give some examples of classes of Hamiltonians satisfying our assumptions. To our knowledge, this is the first investigation of the ergodic problem for second-order MFGCs, as well as the first study of second-order MFGCs with state constraints. The problem of MFGCs with state constraints can be especially challenging and, to our knowledge, has only been studied so far in \cite{graber2021note} in the case of deterministic MFGCs, using a ``mild formulation" of Nash equilibrium. In the present setting, by contrast, we derive the existence of classical solutions, relying heavily on the elliptic theory for problems with state constraints going back to \cite{lasry1989nonlinear}.

While many of the arguments in this paper are generalizations of ones found in \cite{kobeissi2022mean,lasry1989nonlinear,porretta2023ergodic}, there is some significant novelty to our analysis beyond the results themselves. In our analysis of the joint distribution of players' states and controls, we must deal with the fact that our controls blow up at the boundary. Furthermore, to obtain a priori estimates without imposing boundedness or restrictive smallness conditions, we use comparison to a fixed control, a method inspired by \cite{kobeissi2022mean} but adapted to a case where the asymptotic behavior near the boundary plays a significant role. Finally, as we are working with a mean field game of controls and the value function does not belong to a standard Banach space, to prove existence for our system, we choose to apply Schauder's fixed-point theorem to an appropriate map defined on a subset of the space of probability measures, where tightness is used to achieve compactness.

One of the largest motivations for studying MFGCs with state constraints is that they arise naturally in economics \cite{achdou2014partial,achdou2014heterogeneous}.
At present, most theoretical results on mean field games do not encompass such models.
With the present work, we hope to take a step toward filling that gap.

In the remainder of this introduction, we provide some basic notation and assumptions, a formal problem statement (see System \eqref{Eq: Ergodic MFG}), and give some motivating examples that satisfy the stated assumptions.
Section \ref{Sec: Fixed-Point Relation} provides a technical lemma that will be crucial to studying the joint distribution of states and controls.
In Section \ref{Sec: HJ} we collect estimates on solutions to ergodic Hamilton-Jacobi equations with state constraints, where a measure $\mu$ appears as a parameter in the data.
In Section \ref{Sec: FP} we state known results concerning Fokker-Planck equations with an invariance condition on the vector field.
In Section \ref{Sec: A Priori Estimates} we prove the crucial a priori estimates on the system \eqref{Eq: Ergodic MFG}.
This is followed by Section \ref{Sec: MFG}, where we state and prove our main results of existence and uniqueness of solutions.

\subsection{Notation \& Preliminaries}

Before we introduce the system of PDE we intend to study, we will need to establish some notation that we will use. First, we will let $\Omega \subseteq \R^n$ be a bounded open set such that $\B$ is $C^2$-smooth, and we will define the following subdomains:

\begin{definition} \label{Def: Subdomains}
    For every $\varepsilon > 0$, we will denote by $\Gamma_\varepsilon$ and $\Omega_\varepsilon$ the sets
    $$\Gamma_\varepsilon \coloneqq \{x \in \overline{\Omega} : d(x,\B) \leq \varepsilon\} \hspace{1cm} \Omega_\varepsilon \coloneq \Omega \setminus \Gamma_\varepsilon.$$
    Furthermore, we will use $\n(\cdot)$ to denote the unit outward normal vector and $d(\cdot)$ to denote a function in $C^2(\overline{\Omega})$ that is positive in $\Omega$ and coincides with the oriented distance
    $$d_\Omega(x) = 
    \begin{cases}
        d(x,\B), &x \in \overline{\Omega} \\
        -d(x,\B), &x \notin \Omega
    \end{cases}$$
    in $\Gamma_{\varepsilon_0}$ for some $\varepsilon_0 > 0$.
\end{definition}

Next, as the study of the joint distribution of player states and controls is fundamental to our analysis, we will need to discuss the space of measures we will consider.

\begin{definition} \label{Def: Space of Measures}
    We will define $\mathcal{P}_{q'}(\overline{\Omega} \times \R^n)$ to be the set of Borel probability measures $\nu$ on $\overline{\Omega} \times \R^n$ such that $\Lambda_{q'}(\nu) < \infty$, where we define $\Lambda_{r}$, as in \cite{kobeissi2022mean,kobeissi2022classical}, by
    $$\Lambda_r(\mu) = 
    \begin{cases}
    \left(\int_{\Omega \times \R^n} |\alpha|^r d\mu(x,\alpha)\right)^\frac{1}{r}, &1 \leq r < \infty \\
    \sup\{|\alpha| : (x,\alpha) \in \operatorname{supp} \mu\}, &r = \infty
    \end{cases}$$
    for $\mu = (I,a) \# m$.
\end{definition}
By the Riesz representation theorem, the space of all signed regular Borel measures on \(\overline{\Omega} \times \bb{R}^n\) is isometrically isomorphic to the dual of continuous functions on \(\overline{\Omega} \times \bb{R}^n\) that vanish at infinity.
The space $\mathcal{P}_{q'}(\overline{\Omega} \times \R^n)$ thus inherits the weak\(^*\) topology from this space, and unless otherwise stated \(\mu_k \to \mu\) means that \(\mu_k\) converges to \(\mu\) with respect to this topology.

We denote by $v \otimes w$ the tensor product of vectors: if $v \in \R^n$ and $y \in \R^m$, then $v \otimes w$ is the $n \times m$ matrix given by $(v \otimes w)_{ij} = v_iw_j$. Additionally, we write $f(x) = O(g(x))$ (as $x \to x_0$, or $x \to \B$) to say that we have $|f(x)| \leq Cg(x)$ (for $x$ near $x_0$, or $x$ near $\B$) for some constant $C > 0$, and we write $f(x) = o(g(x))$ to mean that $\frac{f(x)}{g(x)} \to 0$.

As for function spaces, we will use the standard notation of $W^{k,r}$ to denote the Sobolev space of $k$-times weakly differentiable functions whose $j$th-order derivatives are $r$-integrable for $0 \leq j \leq k$, and we will write $W^{k,r}_{loc}(\Omega)$ for the set of functions belonging to $W^{k,r}(K)$ for all $K \subset\subset \Omega$. Additionally, for a non-negative integer $k$ and a fraction $\beta \in (0,1)$, we will use $C^{k+\beta}$ to denote the space of $k$-times differentiable functions whose $j$th-order derivatives are $\beta$-H\"{o}lder continuous for all $j \leq k$.

Aside from these preliminaries, we specify that the constant $C$ appearing in many results denotes a generic constant that may change from line to line but depends only on the constants in the assumptions.

\subsection{The System of PDE \& Its Interpretation}

In this article, we will consider the second-order ergodic MFGC system
\begin{equation} \label{Eq: Ergodic MFG}
    \begin{cases}
        -\sigma\Delta u + H(D_xu,\mu) + \rho = F(\mu,x), &x \in \Omega \\
        \sigma\Delta m + \nabla \cdot (mD_pH(D_x u,\mu)) = 0, &x \in \Omega \\
        \mu = (I,-D_pH(D_xu,\mu)) \# m, & \\
        m \geq 0, \hspace{1cm} \int_\Omega m dx = 1, \hspace{1cm} \underset{d(x) \to 0}{\lim} u(x) = \infty
    \end{cases}
\end{equation}

\begin{definition}
    We will say $(u,\rho,m,\mu) \in W^{2,r}_{loc}(\Omega) \times \R \times \mathcal{P}(\Omega) \times \mathcal{P}_{q'}(\overline{\Omega} \times \R^n)$ is a solution to \eqref{Eq: Ergodic MFG} if $(u,\rho)$ solves the Hamilton-Jacobi equation a.e., $m$ satisfies the Fokker-Planck equation in the sense of Definition \ref{Def: Solution to FPK}, and $\mu$ satisfies the fixed-point relation.
\end{definition}

Solutions to this system of PDE correspond to Nash equilibria for a mean field game in which a generic agent's state is given by the SDE
$$dX_t = a(X_t)dt + \sqrt{2\sigma}dB_t, \qquad X_0 = x \in \Omega$$
where the feedback control is constrained to the set
$$\mathcal{A} = \left\{a \in C^0(\Omega;\R^n) : P(x \in \Omega : X_t \in \Omega \quad \forall t > 0) = 1\right\}.$$
As we will prove in Section \ref{Sec: A Priori Estimates}, for a given distribution $\mu$, the solution to the Hamilton-Jacobi equation corresponds to the following optimization problem:
$$\rho = \lim_{T \to \infty} \frac{1}{T} \inf_{a \in \mathcal{A}} E \int_0^T \left(L(a(X_t),\mu) + F(\mu,X_t)\right) dt$$
and
$$u(x) = \inf_{a \in \mathcal{A}} E \left[\int_0^{\theta_a} \left(L(a(X_t),\mu) + F(\mu,X_t)\right) dt + u(X_{\theta_a}) - \theta_a\rho\right]$$
where $\theta_a$ represents a stopping time that is bounded by some $T \geq 0$ which does not depend on the control. In the case of a Nash equilibrium, the probability density $m$ must coincide with the stationary invariant measure associated to the optimal trajectory. Additionally, as this is a mean field game of controls, when the system is in equilibrium, the optimal feedback control given $\mu$ must correspond to $\mu$ itself, resulting in an additional fixed-point problem for $\mu$.

\subsection{Assumptions}

To prove the well-posedness of our system, we will make the following assumptions.
The constants $C_0,q,\widetilde{q}$ and the functions $f_1,f_2,f_3$ listed below are fixed independent of the data.
\begin{assumption} \label{A: F}
    The function $F: \mathcal{P}_{q'}(\overline{\Omega} \times \R^n) \to W^{1,\infty}(\Omega)$ is continuous with $\underset{\mu}{\sup} \|F(\mu,\cdot)\|_{W^{1,\infty}(\Omega)} \leq C_0$ for some constant $C_0 > 0$. Furthermore, we have
    $$\int_\Omega (F(\mu_1,x) - F(\mu_2,x))d(m_1 - m_2) \geq 0$$
    where $m_i$ is the first marginal of $\mu_i$.
\end{assumption}
\begin{assumption} \label{A: Hamiltonian}
    The Hamiltonian $H: \R^n \times \mathcal{P}_{q'}(\overline{\Omega} \times \R^n) \to \R$ is differentiable and strictly convex with respect to the first variable $p$. Furthermore, $H$ and $D_pH$ are continuous on $\R^n \times \mathcal{P}_{q'}(\overline{\Omega} \times \R^n)$, where $\mathcal{P}_{q'}(\overline{\Omega} \times \R^n)$ is endowed with the weak-* topology.
\end{assumption}
\begin{assumption} \label{A: Bounds for H}
    The function $G(p,\mu) \coloneq H(p,\mu) - f_1(\mu)|p + f_2(\mu)|^q$ satisfies
    $$|G(p,\mu)| \leq (|p|^{\widetilde{q}} + 1)f_3(\mu)$$
    for some $q \in (1,2]$, some $0 \leq \widetilde{q} < 1$, and some functions $f_1 > 0, f_2: \mathcal{P}_{q'}(\overline{\Omega} \times \R^n) \to \R^n$, and $f_3 \geq 0$ which send sets in $\mathcal{P}_{q'}(\overline{\Omega} \times \R^n)$ that are bounded with respect to $\Lambda_{q'}$ into compact subsets of $(0,\infty),\R^n$, and $[0,\infty)$, respectively.
\end{assumption}
\begin{assumption} \label{A: Lagrangian}
    The Lagrangian $L: \R^n \times \mathcal{P}_{q'}(\overline{\Omega} \times \R^n) \to \R$ defined by
    \begin{equation}
        L(\alpha,\mu) = \sup_p \{-\alpha \cdot p - H(p,\mu)\}
    \end{equation}
    is strictly convex with respect to $\alpha$.
\end{assumption}
\begin{assumption} \label{A: LL monotonicity}
    For all $\mu_1,\mu_2 \in \mathcal{P}_{q'}(\overline{\Omega} \times \R^n)$, we have
    $$\int_{\Omega \times \R^n} (L(\alpha,\mu_1) - L(\alpha,\mu_2))d(\mu_1-\mu_2)(x,\alpha) \geq 0.$$
\end{assumption}
\begin{assumption} \label{A: Lower bound for L}
    $L(\alpha,\mu) \geq \frac{1}{C_0}|\alpha|^{q'} - C_0\left(1 + \Lambda_{q'}(\mu)^{q'}\right)$, where $q' = \frac{q}{q-1}$.
\end{assumption}
\begin{assumption} \label{A: Bound for L}
    $|L(\alpha,\mu)| \leq C_0\left(1 + |\alpha|^{q'} + \Lambda_{q'}(\mu)^{q'}\right)$.
\end{assumption}
\begin{assumption} \label{A: Regularity of G}
    There exists some
    $$0 < \widetilde{\alpha} < \begin{cases}
        \frac{1}{q+1}, &1 <q \leq \frac{3}{2} \\
        \frac{2(q-1)}{q+1}, & \frac{3}{2} < q < 2
    \end{cases}$$
    so that for all $p_1,p_2 \in \R^n$ and $\mu \in \mathcal{P}_{q'}(\overline{\Omega} \times \R^n)$,
    $${|G(p_1,\mu) - G(p_2,\mu)| \leq f_3(\mu)|p_1 - p_2|^{\widetilde{\alpha}}}.$$
\end{assumption}

In the case that $1 < q < 2$, we will make the following assumption:
\begin{assumption} \label{A: Asymptotics of DpH}
    Assume $q \in (1,2)$. Given $\mu \in \mathcal{P}_{q'}(\overline{\Omega} \times \R^n)$, $\theta_1 > 0$, and $v \in W^{1,\infty}_{loc}(\Omega;\R^n)$ with asymptotic expansion
    \begin{equation}
        v(x) = (f_1(\mu)(q-1)d(x))^{1-q'}[\n(x) + O(d(x)^{\theta_1})],
    \end{equation}
    there exist $\delta_0,\theta_2 > 0$ such that $D_pH(v(x),\mu) \in W^{1,\infty}_{loc}(\Gamma_{\delta_0};\R^n)$ and for $x \in \Gamma_{\delta_0}$,
    $$D_pH(v(x),\mu) = \frac{q'}{d(x)}[\n(x) + O(d(x)^{\theta_2})].$$
    Furthermore, if
    $$\operatorname{Jac}(v(x)) = f_1(\mu)^{1-q'}(q-1)^{-q'}d(x)^{-q'}[\n(x) \otimes \n(x) + O(d(x)^{\theta_1})],$$
    then
    $$\nabla \cdot (D_pH(v(x),\mu)) = \frac{q'}{d(x)^2}[1 + O(d(x)^{\theta_2})]$$
    and $b = D_pH(v(x),\mu)$ satisfies
    \begin{equation} \label{Eq: Invariance Constraint 1}
        \begin{cases}
            \Delta d - b \cdot D_xd \geq \frac{1}{d} - Cd \qquad \text{ for } x \in \Gamma_{\delta_0} \text{ for some } C > 0 \\
            \operatorname{Jac} b \geq -Cd^{\gamma_0-2}I  \qquad \text{ for } x \in \Gamma_{\delta_0} \text{ for some } C > 0, \gamma_0 > 0
        \end{cases}
    \end{equation}
\end{assumption}

In the case $q = 2$, we will replace A\ref{A: Asymptotics of DpH} with the following less general assumption:
\begin{assumption} \label{A: q=2}
    The Hamiltonian $H$ takes the form
    $$H(p,\mu) = \psi(\mu)|p + \varphi(\mu)|^2 + V(\mu).$$
\end{assumption}

In Section \ref{Sec: Examples}, we will discuss some examples of Hamiltonian-Lagrangian pairs satisfying A\ref{A: Hamiltonian}-\ref{A: Asymptotics of DpH} that will help to motivate our analysis.

\subsection{Properties of the Lagrangian and Hamiltonian}
\label{Sec: Hamiltonian}

Before we start our analysis of \eqref{Eq: Ergodic MFG}, we must discuss some properties of $H$ and $L$ that will be useful in later sections. In \cite{kobeissi2022mean}, the author used properties of convex functions to obtain regularity and bounds for the Hamiltonian from properties of the Lagrangian. In this section, we recall these bounds and note that nearly identical arguments can be used in our case to prove similar regularity results for the Lagrangian.

\begin{lemma} \label{Lem: Regularity of L}
    Under assumptions A\ref{A: Hamiltonian}-\ref{A: Bounds for H}, the Lagrangian $L(\alpha,\mu)$ is differentiable with respect to $\alpha$, and $L$ and $D_\alpha L$ are continuous on $\R^n \times \mathcal{P}_{q'}(\overline{\Omega} \times \R^n)$. 
\end{lemma}

\begin{lemma} \label{Lem: Properties of H}
    Under assumptions A\ref{A: Lagrangian} and A\ref{A: Lower bound for L}-\ref{A: Bound for L}, up to a new choice of $C_0$, we have
    \begin{equation}
        |D_pH(p,\mu)| \leq C_0(1 + |p|^{q-1} + \Lambda_{q'}(\mu))
        \label{Eq: Bound for DpH-theta}
    \end{equation}
    \begin{equation}
        |H(p,\mu)| \leq C_0(1 + |p|^{q} + C_0\Lambda_{q'}(\mu)^{q'})
        \label{Eq: Bound for H-theta}
    \end{equation}
    \begin{equation}
        p \cdot D_pH(p,\mu) - H(p,\mu) \geq \frac{1}{C_0}|p|^q - C_0(1 + \Lambda_{q'}(\mu)^{q'})
        \label{Eq: Convexity for H-theta}
    \end{equation}
    for all $p \in \R^n$ and $\mu \in \mathcal{P}_{q'}(\overline{\Omega} \times \R^n)$.
\end{lemma}

\begin{remark}
    Finally, as it will be important to our analysis in Section \ref{Sec: HJ}, we observe that by the convexity of our Hamiltonian, for every $\mu \in \mathcal{P}_{q'}(\overline{\Omega} \times \R^n)$, we have
    \begin{equation} \label{Eq: First Consequence of Convexity}
        \sup_p \{H(\theta p,\mu) - \theta H(p,\mu)\}  = (1-\theta)H(0,\mu) \to 0 \qquad \text{as } \theta \to 1
    \end{equation}
\end{remark}

\subsection{Motivating Examples}
\label{Sec: Examples}

We conclude our introduction by considering some motivating examples.
\begin{example}
    First, we consider the Hamiltonian
    $$H(p,\nu) = |p + \varphi(\nu)|^q + V(\nu)$$
    where $\varphi, V$ are continuous on $\mathcal{P}_{q'}(\overline{\Omega} \times \R^n)$,
    \begin{equation} \label{Eq: Monotonicity for Example 1}
        (\varphi(\nu_1) - \varphi(\nu_2)) \cdot \int_{\Omega \times \R^n} \alpha  d(\nu_1 - \nu_2)(x,\alpha) \geq 0,
    \end{equation}
    $|\varphi(\nu)| \leq C(1 + \Lambda_{q'}(\nu)^{q'-1})$, and $|V(\nu)| \leq C(\Lambda_{q'}(\nu)^{q'} + 1)$.
\end{example}
Cf.~\cite[Section 3.1]{gomes2016extended}.
One can check that the model found in \cite{chan2015bertrand} (see also \cite{graber2018existence,graber2018variational}) has this type of Hamiltonian.

For this Hamiltonian, our associated Lagrangian is
$$L(\alpha,\nu) = \left(q^{1-q'} - q^{-q'}\right)|\alpha|^{q'} + \alpha \cdot \varphi(\nu) - V(\nu).$$
That $H$ and $L$ satisfy A\ref{A: Hamiltonian}-\ref{A: Asymptotics of DpH} is straightforward to check. For example, if
$$v(x) = ((q-1)d(x))^{1-q'}[\n(x) + O(d(x)^{\theta_1})]$$
and
$$\operatorname{Jac}(v(x)) = (q-1)^{-q'}d(x)^{-q'}[\n(x) \otimes \n(x) + O(d(x)^{\theta_1})],$$
then
$$
\begin{aligned}
    \operatorname{Jac}(D_pH(v,\mu)) &= q|v + \varphi(\mu)|^{q-2}\operatorname{Jac}(v) + q(q-2)|v + \varphi(\mu)|^{q-4} (v + \varphi(\mu)) \otimes \left((\operatorname{Jac}(v))(v + \varphi(\mu))\right) \\
    &= \frac{q'}{d^2}\left(\n(x) \otimes \n(x) + O(d(x)^{\theta_1})\right),
\end{aligned}
$$
which implies
$$\nabla \cdot (D_pH(v,\mu)) = \sum_{i=1}^n \langle(\operatorname{Jac}(D_pH(v,\mu)))e_i,e_i\rangle = \frac{q'}{d^2}\left(1 + O(d(x)^{\theta_1})\right)$$
and
$$\langle(\operatorname{Jac}(D_pH(v,\mu)))\xi,\xi\rangle = \frac{q'}{d^2}\left[|\langle\xi,\n(x)\rangle|^2 + O(d(x)^{\theta_1})\right] \geq - Cd^{\theta_1 - 2}.$$

\begin{example}
    Another potential application would be to Hamiltonians of the form
    $$H(p,\nu) = \psi(\nu)|p|^q + V(\nu), \qquad L(\alpha,\nu) = \left(q^{1-q'} - q^{-q'}\right)\psi(\nu)^{1-q'}|\alpha|^{q'} - V(\nu),$$
    where $\psi, V$ are continuous on $\mathcal{P}_{q'}(\overline{\Omega} \times \R^n)$, $\frac{1}{C} \leq \psi \leq C$ for some $C > 0$, $|V(\nu)| \leq C(1 + \Lambda_{q'}(\nu)^{q'})$, and
    \begin{equation} \label{Eq: Monotonicity for Example 2}
        (\psi(\nu_1) - \psi(\nu_2))(\Lambda_{q'}(\nu_1) - \Lambda_{q'}(\nu_2)) \leq 0.
    \end{equation}
\end{example}
Again, it is straightforward to check that $H$ and $L$ satisfy A\ref{A: Hamiltonian}-\ref{A: Asymptotics of DpH}. For example, if
$$v(x) = (\psi(\mu)(q-1)d(x))^{1-q'}[\n(x) + O(d(x)^{\theta_1})]$$
and
$$\operatorname{Jac}(v(x)) = \psi(\mu)^{1-q'}(q-1)^{-q'}d(x)^{-q'}[\n(x) \otimes \n(x) + O(d(x)^{\theta_1})],$$
then
$$
\begin{aligned}
    \operatorname{Jac}(D_pH(v,\mu)) &= q\psi(\mu)|v|^{q-2}\operatorname{Jac}(v) + q(q-2)\psi(\mu)|v|^{q-4} v \otimes \left((\operatorname{Jac}(v))v\right) \\
    &= \frac{q'}{d^2}\left(\n(x) \otimes \n(x) + O(d(x)^{\theta_1})\right)
\end{aligned}
$$
as before.

\section{Fixed-Point Relation in $\mu$}
\label{Sec: Fixed-Point Relation}

In any study of mean field games of controls, it is crucial to analyze the fixed-point relation satisfied by $\mu$. Our analysis will be similar to those in \cite{graber2025mean,kobeissi2022mean}, but we will need to deal with the complications that arise from the fact that our controls blow up as $x$ approaches the boundary.

\begin{lemma} \label{Lem: Existence & Bound for mu}
    Assume A\ref{A: Hamiltonian}-\ref{A: Bound for L} hold. Given $(p,m) \in C^0(\Omega;\R^n) \times \mathcal{P}(\Omega)$ with $\underset{d(x) \to 0}{\lim} |p|^{q-1}d(x) = \gamma > 0$ and $\int_\Omega d^{-q'}dm < \infty$, we have the following:
    \begin{enumerate}
        \item[1)] If $\mu$ satisfies
        \begin{equation}
            \mu = (I,-D_pH(p(\cdot),\mu)) \# m,
            \label{Eq: mu}
        \end{equation}
        then we have
        \begin{equation}
            \Lambda_{q'}(\mu)^{q'} \leq 4C_0^2 + \frac{(q')^{q-1}(2C_0)^q}{q}\|p\|_{L^q(m)}^q.
            \label{Eq: Bound for Lambda-q'}
        \end{equation}
        \item[2)] There is at most one $\mu \in \mathcal{P}_{q'}(\overline{\Omega} \times \R^n)$ satisfying \eqref{Eq: mu}.
        \item[3)] There exists a unique $\mu \in \mathcal{P}_{q'}(\overline{\Omega} \times \R^n)$ satisfying \eqref{Eq: mu}.
    \end{enumerate}
\end{lemma}

\begin{proof}
    1) The proof is nearly identical to those found in \cite{graber2025mean,kobeissi2022mean}.

    2) The proof of uniqueness is nearly identical to those found in \cite{cardaliaguet2018mean,graber2025mean}.

    3) Take $\delta > 0$ such that $|p| \leq (\gamma^{1/(q-1)} + 1)d(x)^{-1/(q-1)}$ on $\Gamma_\delta$ and fix $C' \geq |p|$ on $\Omega_\delta$. Let $(p_k)_{k \in \N}$ be a sequence in $C^0(\overline{\Omega};\R^n)$ converging to $p$ locally uniformly, which we can assume satisfies the same inequalities on $\Gamma_\delta$ and $\Omega_\delta$ for all $k$. By arguments found in \cite{graber2025mean,kobeissi2022mean}, for each $k$, there exists a unique fixed-point
    $$\mu_k = (I,-D_pH(p_k,\mu_k)) \# m.$$
    By Part 1, we get
    $$
    \begin{aligned}
        \Lambda_{q}(\mu_k)^{q'} &\leq 4C_0 + \frac{(q')^{q-1}(2C_0)^q}{q}\int_\Omega |p_k|^q dm \\
        &\leq 4C_0 + \frac{(q')^{q-1}(2C_0)^q}{q} \left((C')^q + (\gamma^{1/(q-1)}+1)^q\int_{\Gamma_\delta} d(x)^{-q'}dm\right) \\
        &\leq C.
    \end{aligned}
    $$
    By tightness, there is some $\mu \in \mathcal{P}_{q'}(\overline{\Omega} \times \R^n)$ so that $\mu_k \to \mu$, passing to a subsequence if necessary. Moreover, by the continuity of $D_pH$, it follows that $\mu$ satisfies \eqref{Eq: mu} (see the proof of Lemma \ref{Lem: Continuity of mu} for more details). Thus, the result follows by uniqueness.
\end{proof}

\begin{lemma} \label{Lem: Continuity of mu}
    Assume A\ref{A: Hamiltonian}-\ref{A: Bound for L} hold. Let $(p_k,m_k)$ be a sequence in $C^0(\Omega;\R^n) \times \mathcal{P}(\Omega)$ such that
    \begin{enumerate}
        \item $|p_k| \leq Cd(x)^{-\frac{1}{q-1}}$ and $m_k \leq Cd(x)^{q'}$ for some $C \geq 0$ independent of $k$;
        \item $p_k \to p$ and $m_k \to m$ a.e. in $\Omega$.
    \end{enumerate}
    Then $\mu_k \to \mu$, where $\mu_k$ and $\mu$ are the fixed-points corresponding to $(p_k,m_k)$ and $(p,m)$, respectively.
\end{lemma}

\begin{proof}
    By tightness, we get that there is a measure $\widetilde{\mu} \in \mathcal{P}_{q'}(\overline{\Omega} \times \R^n)$ such that, passing to a subsequence if necessary, $\mu_k \to \widetilde{\mu}$. Fixing $x_0 \in \Omega$, we get that for $k \in \N$,
    $$
    \begin{aligned}
        &W_1(\mu_k,(I,-D_pH(p,\widetilde{\mu})) \# m) \\
        &= \sup_{Lip(\phi) \leq 1} \left[\int_\Omega \phi(x,-D_pH(p_k,\mu_k)) dm_k - \int_\Omega \phi(x,-D_pH(p,\mu)) dm\right] \\
        &= \sup_{Lip(\phi) \leq 1} \bigg[\int_\Omega (\phi(x,-D_pH(p_k,\mu_k))-\phi(x,-D_pH(p,\mu))) dm_k \\
        &\qquad+ \int_\Omega (\phi(x,-D_pH(p,\mu)) - \phi(x_0,0)) d(m_k - m)\bigg] \\
        &\leq \int_\Omega |D_pH(p_k,\mu_k))-D_pH(p,\mu)| dm_k + \int_\Omega (|x-x_0| +|D_pH(p,\mu)|) |m_k - m| dx \\
        &\to 0
    \end{aligned}
    $$
    as $k \to \infty$ by the dominated convergence theorem. Thus, the conclusion follows by uniqueness.
\end{proof}

\section{The Hamilton-Jacobi Equation}
\label{Sec: HJ}

In this section, as in \cite{lasry1989nonlinear}, we will consider the ergodic system
\begin{equation} \label{Eq: Ergodic HJB}
    \begin{cases}
        -\Delta u + H(D_xu,\mu) + \rho = F(\mu,x), &x \in \Omega \\
        \underset{d(x) \to 0}{\lim} u(x) = \infty
    \end{cases}
\end{equation}
for some fixed $\mu \in \mathcal{P}_{q'}(\overline{\Omega} \times \R^n)$ by first analyzing the discounted problem
\begin{equation} \label{Eq: Discounted HJB}
    \begin{cases}
        -\Delta u_\lambda + H(D_xu_\lambda,\mu) + \lambda u_\lambda = F(\mu,x), &x \in \Omega \\
        \underset{d(x) \to 0}{\lim} u_\lambda(x) = \infty
    \end{cases}
\end{equation}
and taking $\lambda \to 0$. We observe that there is no loss of generality in assuming $\sigma = 1$, as otherwise we could replace $H$, $\rho$, $F$, and $\lambda$ by $\sigma^{-1}H$, $\sigma^{-1}\rho$, $\sigma^{-1}F$, and $\sigma^{-1}\lambda$, respectively. Furthermore, for simplicity of presentation, we will assume in this section that $f_2(\mu) = 0$ in all except for the proof of Lemma \ref{Lem: Asymptotic Behavior of Derivatives}. Otherwise, we would note that $u$ is a solution to \eqref{Eq: Ergodic HJB} (resp. \eqref{Eq: Discounted HJB}) if and only if $v = u + x \cdot f_2(\mu)$ is a solution to
\begin{equation*}
    \begin{cases}
        -\Delta v + H(D_xv - f_2(\mu),\mu) + \rho = F(\mu,x), &x \in \Omega \\
        (\text{resp. } -\Delta v + H(D_xv - f_2(\mu),\mu) + \lambda v = F(\mu,x) - \lambda x \cdot f_2(\mu), &x \in \Omega) \\
        \underset{d(x) \to 0}{\lim} v(x) = \infty
    \end{cases}
\end{equation*}
and we would perform much of our analysis on $v$ instead of $u$.
This would not change the estimates derived in this section.

\subsection{Well-Posedness}

As in \cite{lasry1989nonlinear}, we prove the well-posedness of \eqref{Eq: Ergodic HJB} by taking a sequence of solutions to the discounted problem \eqref{Eq: Discounted HJB}, letting $\lambda \to 0$. For this, we will require the following generalization of \cite[Theorem II.1]{lasry1989nonlinear}. The proof is similar, but we include it here for completeness.

\begin{lemma} \label{Lem: Discounted HJB}
    Assume A\ref{A: F}-\ref{A: Lagrangian} and A\ref{A: Lower bound for L}-\ref{A: Bound for L} hold. Given $\lambda,r > 0$, there exists a unique solution $u_\lambda \in W^{2,r}_{loc}(\Omega)$ of \eqref{Eq: Discounted HJB}. In addition, we have
    \begin{equation} \label{Eq: Boundary Behavior of u}
        \begin{cases}
            \underset{d(x) \to 0}{\lim} u_\lambda(x)d(x)^\frac{2-q}{q-1} = (q-1)^\frac{q-2}{q-1}(2-q)^{-1}f_1(\mu)^{-\frac{1}{q-1}}, &1 \leq q < 2 \\
            \underset{d(x) \to 0}{\lim} \frac{u_\lambda(x)}{|\ln{d(x)}|} = \frac{1}{f_1(\mu)}, &q = 2
        \end{cases}
    \end{equation}
\end{lemma}

\begin{proof}
    Following the approach in \cite[Theorem II.1]{lasry1989nonlinear}, for $\delta > 0$ and $\varepsilon \in (0,1]$, define $w_{\varepsilon,\delta}^1, w_{\varepsilon,\delta}^2$ by
    $$w_{\varepsilon,\delta}^1 =
    \begin{cases}
        (\widetilde{f} + \varepsilon)(d - \delta)^{-\beta} + C_{\varepsilon,\lambda}, &1 \leq q < 2 \\
        -(\widetilde{f} + \varepsilon)\ln(d - \delta) + C_{\varepsilon,\lambda}, &q = 2
    \end{cases}$$
    on $\Omega_\delta \coloneqq \{x \in \Omega : d(x,\B) > \delta\}$ and
    $$w_{\varepsilon,\delta}^2 =
    \begin{cases}
        (\widetilde{f} - \varepsilon)(d + \delta)^{-\beta} - C_{\varepsilon,\lambda}, &1 \leq q < 2 \\
        -(\widetilde{f} - \varepsilon)\ln(d + \delta) - C_{\varepsilon,\lambda}, &q = 2
    \end{cases}$$
    on $\Omega^\delta \coloneqq \{x \in \R^n : d(x,\Omega) < \delta\}$, where $\beta = \frac{2-q}{q-1}$,
    $$\widetilde{f} = 
    \begin{cases}
        \beta^{-1}(\beta+1)^\frac{1}{(q-1)}f_1(\mu)^{-\frac{1}{(q-1)}}, &1 \leq q < 2 \\
        \frac{1}{f_1(\mu)}, &q = 2
    \end{cases}
    $$
    and $C_{\varepsilon,\lambda}$ is a constant to be chosen. Now for $R > 0$ and $\lambda > 0$, define $u_{R,\lambda}$ to be the unique solution to
    $$\begin{cases}
        -\Delta u_{R,\lambda} + H(D_xu_{R,\lambda},\mu) + \lambda u_{R,\lambda} = F(\mu,x), &\text{in } \Omega \\
        u_{R,\lambda} = w_{\varepsilon,1/R}^2, &\text{on } \B
    \end{cases}$$
    for a fixed $\varepsilon > 0$ (well posedness follows from classical theory, e.g., \cite{amann1978some}). Then there is some $C_\varepsilon > 0$ such that for $C_{\varepsilon,\lambda} \geq \lambda^{-1}C_\varepsilon(1 + f_3(\mu))$, we get that $w_{\varepsilon,\delta}^1$ is a supersolution of \eqref{Eq: Discounted HJB} and $w_{\varepsilon,\delta}^2$ is a subsolution. Hence, the maximum principle (see \cite{bony1967principe,lions1983remark}) gives
    \begin{equation} \label{Eq: Bounding u by sub/super solutions}
        w_{\varepsilon,1/R}^2 \leq u_{R,\lambda} \leq u_{R',\lambda} \leq w_{\varepsilon',0}^1
    \end{equation}
    for all $0 < R < R'$ and $\varepsilon' > 0$. Thus, by Theorem \ref{Thm: Gradient Estimate (Given mu)}, for each $\lambda > 0$, $r > 0$, and $K \subset\subset \Omega$, we get uniform bounds for $u_{R,\lambda}$ in $W^{2,r}(K)$. Using a diagonal argument, this gives a subsequence converging to some $u_\lambda$ in $W^{2,r}_{loc}(\Omega)$, which is a solution to \eqref{Eq: Discounted HJB}.
    
    We now shift to proving uniqueness of solutions. To this end, we first note that for any solution $v$ of \eqref{Eq: Discounted HJB}, the maximum principle gives $v \geq u_{R,\lambda}$. Hence, passing to the limit, we get that $u_\lambda$ is the minimum solution of \eqref{Eq: Discounted HJB}. To build a maximum solution, let $u_{\lambda,\delta}$ be the minimum solution on $\Omega_\delta$ for $\delta > 0$. Then we have
    $$w_{\varepsilon,\delta}^2 \leq u_{\lambda,\delta} \leq w_{\varepsilon,\delta}^1$$
    for all $\varepsilon > 0$ and $u_{\lambda,\delta} \geq u_{\lambda,\delta'}$ for $\delta \geq \delta' > 0$. Passing to the limit as before, we get a solution $\widetilde{u}_\lambda$ of \eqref{Eq: Discounted HJB} such that $w_{\varepsilon,0}^2 \leq \widetilde{u}_\lambda \leq w_{\varepsilon,0}^1$ in $\Omega$. Again, by the maximum principle, every solution $v$ of \eqref{Eq: Discounted HJB} satisfies $v \leq u_{\lambda,\delta}$ for all $\delta > 0$, and hence $v \leq \widetilde{u}_\lambda$. Thus, for all solutions $v$ to \eqref{Eq: Discounted HJB}, we have
    \begin{equation} \label{Eq: Min-Max Solutions}
        w_{\varepsilon,0}^2 \leq u_\lambda \leq v \leq \widetilde{u}_\lambda \leq w_{\varepsilon,0}^1.
    \end{equation}

    To prove the uniqueness of solutions, we need only show that $u_\lambda = \widetilde{u}_\lambda$. To this end, note that \eqref{Eq: Min-Max Solutions} gives
    $$\lim_{d(x) \to 0} \frac{\widetilde{u}_\lambda}{u_\lambda} = 1.$$
    Thus, for all $\theta \in (0,1)$ close to 1, $u_\lambda > \theta \widetilde{u}_\lambda - \underset{p}{\sup}\{H(\theta p,\mu) - \theta H(p,\mu)\}/\lambda$ in a neighborhood of $\B$. Additionally, letting $w = \theta \widetilde{u}_\lambda - \underset{p}{\sup}\{H(\theta p,\mu) - \theta H(p,\mu)\}/\lambda$, we get
    $$
    \begin{aligned}
        -\Delta w + H(D_xw,\mu) + \lambda w - \theta F(\mu,x) &= H(\theta D_x\widetilde{u}_\lambda,\mu) - \theta H(D_x\widetilde{u}_\lambda,\mu) - \underset{p}{\sup}\{H(\theta p,\mu) - \theta H(p,\mu)\} \\
        &\leq 0
    \end{aligned}
    $$
    and so $w \leq u_\lambda$. Letting $\theta \to 1$ shows that $u_\lambda = \widetilde{u}_\lambda$ by \eqref{Eq: First Consequence of Convexity}.
\end{proof}

Next, we will require the following generalization of \cite[Theorem II.2]{lasry1989nonlinear}, which is proven by modifying the proof of Theorem \ref{Lem: Discounted HJB} as in \cite{lasry1989nonlinear}. This proof, we omit.

\begin{theorem} \label{Thm: HJB with Locally Bounded Righthand Side}
    Assume A\ref{A: F}-\ref{A: Lagrangian} and A\ref{A: Lower bound for L}-\ref{A: Bound for L} hold and suppose $g \in L^\infty_{loc}(\Omega)$ such that $g$ is bounded from below and $\underset{d(x) \to 0}{\lim} g(x)d(x)^q = 0$. Given $\lambda,r > 0$, there exists a unique solution $u_\lambda \in W^{2,r}_{loc}(\Omega)$ of
    \begin{equation} \label{Eq: HJB with Locally Bounded Righthand Side}
        \begin{cases}
            -\Delta u_\lambda + H(D_xu_\lambda,\mu) + \lambda u_\lambda = g, &x \in \Omega \\
            \underset{d(x) \to 0}{\lim} u_\lambda(x) = \infty
        \end{cases}
    \end{equation}
    In addition, we have
    \begin{equation} \label{Eq: Boundary Behavior 2}
        \begin{cases}
            \underset{d(x) \to 0}{\lim} u_\lambda(x)d(x)^\frac{2-q}{q-1} = (q-1)^\frac{q-2}{q-1}(2-q)^{-1}f_1(\mu)^{-\frac{1}{q-1}}, &1 \leq q < 2 \\
            \underset{d(x) \to 0}{\lim} \frac{u_\lambda(x)}{|\ln{d(x)}|} = \frac{1}{f_1(\mu)}, &q = 2
        \end{cases}
    \end{equation}
\end{theorem}

With this, we are ready to prove the well-posedness of \eqref{Eq: Ergodic HJB}.

\begin{theorem} \label{Thm: HJB has a solution}
    Assume A\ref{A: F}-\ref{A: Lagrangian} and A\ref{A: Lower bound for L}-\ref{A: Bound for L} hold. Then the system \eqref{Eq: Ergodic HJB} has a unique solution $(u,\rho) \in W^{2,r}_{loc}(\Omega) \times \R$ for all $1 < r < \infty$. Furthermore, $u$ satisfies
    \begin{equation} \label{Eq: Boundary Behavior 3}
        \begin{cases}
            \underset{d(x) \to 0}{\lim} u(x)d(x)^\frac{2-q}{q-1} = (q-1)^\frac{q-2}{q-1}(2-q)^{-1}f_1(\mu)^{-\frac{1}{q-1}}, &q < 2 \\
            \underset{d(x) \to 0}{\lim} \frac{u(x)}{|\ln{d(x)}|} = \frac{1}{f_1(\mu)}, &q = 2
        \end{cases}
    \end{equation}
\end{theorem}

\begin{proof}
    By \eqref{Eq: Min-Max Solutions}, we have
    $$\frac{\widetilde{f} - \varepsilon}{d^\beta} - \frac{C_\varepsilon(1 + f_3(\mu))}{\lambda} \leq u_\lambda \leq \frac{\widetilde{f} + \varepsilon}{d^\beta} + \frac{C_\varepsilon(1 + f_3(\mu))}{\lambda}$$
    for $1 \leq q < 2$ and
    $$-(\widetilde{f} - \varepsilon)\ln{d} - \frac{C_\varepsilon(1 + f_3(\mu))}{\lambda} \leq u_\lambda \leq - (\widetilde{f} + \varepsilon)\ln{d} + \frac{C_\varepsilon(1 + f_3(\mu))}{\lambda}$$
    for $q = 2$. This implies that $\lambda u_\lambda$ bounded from below and in $L^\infty(K)$ for all $K \subset\subset \Omega$, uniformly in $\lambda \in (0,1]$. By Theorem \ref{Thm: Gradient Estimate (Given mu)}, letting $v_\lambda \coloneqq u_\lambda - u_\lambda(x_0)$ for some fixed $x_0 \in \Omega$, we get that for all $K \subset\subset \Omega$, $v_\lambda$ is bounded in $W^{2,\infty}(K)$ uniformly in $\lambda$.
    
    Note that $v_\lambda$ satisfies
    $$-\Delta v_\lambda + H(D_xv_\lambda,\mu) +\lambda v_\lambda = -\lambda u_\lambda(x_0) + F(\mu,x).$$
    Choosing $C_1 \in (0,\widetilde{f})$ and setting $z = C_1d^{-\beta}$, we get that
    $$-\Delta z + H(D_xz,\mu) + \lambda z \leq -\lambda u_\lambda(x_0) + F(\mu,x)$$
    on $\Omega \setminus \Omega_\delta$ for sufficiently small $\delta$, say $\delta \leq \delta_0$. Also, there is some $M \geq 0$ so that $v_\lambda - C_1d^{-\beta} \geq -M$ on $\Omega_{\delta_0}$. Hence,
    $$v_\lambda \geq -M + C_1d^{-\beta}.$$
    Using our local estimates and a diagonal argument, we get that, up to a subsequence, $\lambda u_\lambda$ converges to some $\rho \in \R$ and $v_\lambda$ converges to some $v$ in $W^{2,r}(K)$ for all $K \subset\subset \Omega$, which solves
    \begin{equation}
        -\Delta v + H(D_xv,\mu) + \rho = F(\mu,x).
    \end{equation}
    Furthermore, $v \geq -M + C_1d^{-\beta}$ and so $\underset{d(x) \to 0}{\lim} v(x) = \infty$.

    Now suppose $(\widetilde{v},\widetilde{\rho})$ satisfies
    $$
    \begin{cases}
        -\Delta\widetilde{v} + H(D_x\widetilde{v},\mu) + \widetilde{\rho} = F(\mu,x) \\
        \underset{d(x) \to 0}{\lim} \widetilde{v}(x) = \infty
    \end{cases}
    $$
    Note that $w_{\varepsilon,\delta}^3 \coloneqq (\widetilde{f} + \varepsilon)(d - \delta)^{-\beta}$ satisfies
    $$-\Delta w_{\varepsilon,\delta}^3 + H(D_xw_{\varepsilon,\delta}^3,\mu) + \widetilde{\rho} \geq F(\mu,x)$$
    in $\Omega_\delta \setminus \Omega_{\delta_0}$ for some $\delta_0 = \delta_0(\varepsilon) > \delta$. Thus, there exist $C, M_\varepsilon \geq 0$ such that
    $$-C \leq \widetilde{v} \leq (\widetilde{f} + \varepsilon)d^{-\beta} + M_\varepsilon.$$
    Now note that
    $$-\Delta\widetilde{v} + H(D_x\widetilde{v},\mu) + \widetilde{v} = g$$
    where $g = \widetilde{v} + F - \widetilde{\rho} \in L^\infty_{loc}(\Omega)$ is bounded from below and satisfies $\underset{d(x) \to 0}{\lim} g(x)d(x)^q \in [0,\infty)$. By Theorem \ref{Thm: HJB with Locally Bounded Righthand Side}, we have
    \begin{equation} \label{Eq: Boundary Behavior 4}
        \begin{cases}
            \underset{d(x) \to 0}{\lim} \widetilde{v}(x)d(x)^\frac{2-q}{q-1} = (q-1)^\frac{q-2}{q-1}(q-2)^{-1}f_1(\mu)^{-\frac{1}{q-1}}, &1 \leq q < 2 \\
            \underset{d(x) \to 0}{\lim} \frac{\widetilde{v}(x)}{|\ln{d(x)}|} = \frac{1}{f_1(\mu)}, &q = 2
        \end{cases}
    \end{equation}

    Now we shift to proving uniqueness. To this end, suppose $(u_1,\rho_1), (u_2,\rho_2)$ are solutions to
    $$
    \begin{cases}
        -\Delta u_i + H(D_xu_i,\mu) + \rho_i = F(\mu,x) \\
        \underset{d(x) \to 0}{\lim} u_i(x) = \infty
    \end{cases}
    $$
    and suppose, without loss of generality, that $\rho_1 < \rho_2$. Then for $\varepsilon > 0$ and $\theta \in (0,1)$,
    $$
    \begin{aligned}
        -\Delta(\theta u_2) + H(D_x(\theta u_2),\mu) + \varepsilon\theta u_2 - \theta F(\mu,x) &= \varepsilon\theta u_2 - \theta\rho_2 + (H(D_x(\theta u_2),\mu) - \theta H(D_xu_2,\mu))
    \end{aligned}
    $$
    By \eqref{Eq: Boundary Behavior 4}, there is some $C_\theta > 0$ so that $\theta u_2 \leq u_1 + C_\theta$ in $\Omega$. Hence,
    $$
    \begin{aligned}
        &-\Delta(\theta u_2) + H(D_x(\theta u_2),\mu) + \varepsilon\theta u_2 \\
        &\qquad \leq \theta F(\mu,x) + \varepsilon u_1 - \rho_1 + (\rho_1 - \theta\rho_2) + (H(D_x(\theta u_2),\mu) - \theta H(D_xu_2,\mu)) + \varepsilon C_\theta \\
        &\qquad \leq F(\mu,x) + C(1 - \theta) + \varepsilon u_1 - \rho_1 + (\rho_1 - \theta\rho_2) + (H(D_x(\theta u_2),\mu) - \theta H(D_xu_2,\mu)) + \varepsilon C_\theta.
    \end{aligned}
    $$
    Choosing $\theta$ close enough to $1$ and $\varepsilon$ close to $0$ (depending on $\theta$), we get that $\theta u_2$ is a subsolution of
    $$-\Delta v + H(D_xv,\mu) + \varepsilon v = \varepsilon u_1 + F(\mu,x) -\rho_1.$$
    Thus, we have $\theta u_2 \leq u_1$ for $\theta$ sufficiently close to $1$. In particular, $u_2 \leq u_1$. However, since $u_2 + C$ satisfies the same equation for all $C \in \R$, this is a contradiction. Therefore, we have $\rho_1 = \rho_2 \eqqcolon \rho$.

    To show that $u_1 = u_2$ (up to a constant), choose $C_1 \in (0,\widetilde{f})$ and choose $\delta > 0$ such that
    $$-\Delta\left(\frac{C_1}{d^\beta}\right) + H\left(D_x\left(\frac{C_1}{d^\beta}\right),\mu\right) \leq F(\mu,x) -\rho$$
    in $\Omega \setminus \Omega_\delta$. Then for $\theta \in (0,1)$ and $w = \theta u_1 + (1-\theta)C_1d^{-\beta}$, we get
    $$
    \begin{aligned}
        -\Delta w + H(D_xw,\mu) &\leq \theta(F(\mu,x) - \rho) + (1-\theta)(F(\mu,x) - \rho) + H\left(D_x\left(\theta u_1 + (1-\theta)\frac{C_1}{d^\beta}\right),\mu\right) \\
        &\qquad - \theta H(D_xu_1,\mu) - (1-\theta)H\left(D_x\left(\frac{C_1}{d^\beta}\right),\mu\right) \\
        &\leq F(\mu,x) - \rho
    \end{aligned}
    $$
    in $\Omega \setminus \Omega_\delta$ by convexity. Since $\underset{d(x) \to 0}{\lim} (w-u_2) = -\infty$, the maximum principle gives us that
    $$\sup_{\Omega \setminus \Omega_\delta} (w-u_2) = \sup_{\B_\delta} (w-u_2)$$
    and hence (letting $\theta \to 1$)
    $$\sup_{\Omega \setminus \Omega_\delta} (u_1-u_2) = \sup_{\B_\delta} (u_1-u_2).$$
    Thus, applying the maximum principle on $\Omega_\delta$, we get
    $$\sup_{\Omega} (u_1-u_2) = \sup_{\B_\delta} (u_1-u_2).$$
    However, applying the maximum principle on $\Omega$, this implies that $(u_1 - u_2) \equiv (u_1 - u_2)(x_0)$.
\end{proof}

\subsection{Gradient Estimate \& Asymptotic Expansions}

Next, we obtain an a priori estimate for the gradient of $u_\lambda$, which immediately gives an estimate for the gradient of $u$. The argument used is similar to the one used for \cite[Theorem IV.1]{lasry1989nonlinear}.

\begin{theorem} \label{Thm: Gradient Estimate (Given mu)}
    Assume A\ref{A: F}-\ref{A: Lagrangian} and A\ref{A: Lower bound for L}-\ref{A: Bound for L} hold and let $u_\lambda$ be a solution to \eqref{Eq: Discounted HJB} for some $\lambda > 0$. Then there is some $C = C(f_1(\mu),f_3(\mu),\Lambda_{q'}(\mu),\Omega,q,\widetilde{q},n) > 0$ so that
    $$|D_xu_\lambda| \leq Cd(x)^{-\frac{1}{q-1}}.$$
\end{theorem}

\begin{proof}
    Let $x_0 \in \Omega$ and $r = \frac{1}{2}d(x_0,\B)$. Now consider $\widetilde{u}_\lambda(x) = r^{\gamma-1}u_\lambda(x_0+rx)$ on $B(0,1)$. Then $\widetilde{u}_\lambda$ solves
    $$-r^{(q-1)\gamma-1}\Delta\widetilde{u}_\lambda + r^{q\gamma}H(r^{-\gamma}D_x\widetilde{u}_\lambda,\mu) + \lambda r^{(q-1)\gamma+1}\widetilde{u}_\lambda = r^{q\gamma}F.$$
    Now define $\varphi \in C^\infty_c(B(0,1))$ satisfying
    \begin{equation}
        \begin{cases}
            0 \leq \varphi \leq 1, \qquad \varphi \equiv 1 \text{ on } B(0,1/2) \\
            |\Delta\varphi| \leq M\varphi^\theta, \qquad |D_x\varphi|^2 \leq M\varphi^{1+\theta}
        \end{cases}
    \end{equation}
    for some $M > 0$ and some $\theta$ to be chosen. We will assume $u$ is smooth to avoid the tedious approximation arguments. Letting $w_\lambda = |D_x\widetilde{u}_\lambda|^2$, we get
    $$
    \begin{aligned}
        &2\lambda r^{(q-1)\gamma + 1}\varphi w_\lambda + r^{(q-1)\gamma-1}\left(w_\lambda\Delta\varphi + 2\varphi|D_{xx}^2\widetilde{u}_\lambda|^2 + \frac{2}{\varphi}D_x\varphi \cdot D_x(\varphi w_\lambda)\right) - 2r^{q\gamma}\varphi D_xF \cdot D_xu_\lambda \\
        &= r^{(q-1)\gamma-1}\left(\Delta(\varphi w_\lambda) + \frac{2|D_x\varphi|^2}{\varphi}w_\lambda\right) + r^{(q-1)\gamma} D_pH(r^{-\gamma}D_x\widetilde{u}_\lambda,\mu) \cdot (w_\lambda D_x\varphi - D_x(\varphi w_\lambda))
    \end{aligned}
    $$
    on $\operatorname{supp}\varphi$. Letting $x_1 \in \operatorname{supp}\varphi$ be a maximum point for $\varphi w$, the maximum principle gives that
    $$
    \begin{aligned}
        2r^{(q-1)\gamma-1}\varphi|D_{xx}^2\widetilde{u}_\lambda|^2 &\leq -r^{(q-1)\gamma-1}w_\lambda\Delta\varphi + 2r^{(q-1)\gamma-1}\frac{|D_x\varphi|^2}{\varphi}w_\lambda + 2C_0r^{q\gamma}\varphi w_\lambda^{\frac{1}{2}} \\
        &\qquad+ r^{(q-1)\gamma} w_\lambda D_pH(r^{-\gamma}D_x\widetilde{u}_\lambda,\mu) \cdot D_x\varphi \\
        &\leq 3Mr^{(q-1)\gamma - 1}\varphi^\theta w_\lambda + MC_0(w_\lambda^{\frac{q+1}{2}} + r^{(q-1)\gamma}(1 + \Lambda_{q'}(\mu)))\varphi^\frac{1+\theta}{2} \\
        &\qquad + 2C_0r^{q\gamma}\varphi w_\lambda^{\frac{1}{2}}.
    \end{aligned}
    $$
    at $x_1$. From the Cauchy-Schwartz inequality, we get
    $$|D_{xx}^2\widetilde{u}_\lambda|^2 \geq \frac{1}{n}(\Delta\widetilde{u})^2 = \frac{1}{n}(r^{\gamma+1}H(r^{-\gamma}D_x\widetilde{u}_\lambda,\mu) + \lambda r^2\widetilde{u}_\lambda - r^{\gamma + 1}F)^2.
    $$
    Combining these results with A\ref{A: Bounds for H}, and using the fact that $\lambda\widetilde{u}_\lambda$ is bounded from below, we get
    $$
    \begin{aligned}
        \varphi w_\lambda^q &\leq C\Bigg(r^{2q\gamma} + \frac{1}{f_1(\mu)^2}r^{2(q-\widetilde{q})\gamma}\varphi w_\lambda^{\widetilde{q}}f_3(\mu)^2 + \frac{1}{f_1(\mu)^2}r^{2q\gamma}\varphi f_3(\mu)^2 + \frac{1}{f_1(\mu)^2}r^{2(q-1)\gamma - 2}\varphi^\theta w_\lambda \\
        &\qquad + \frac{1}{f_1(\mu)^2}\left(r^{(q-1)\gamma - 1}w_\lambda^\frac{q+1}{2} + r^{2(q-1)\gamma - 1}(1 + \Lambda_{q'}(\mu))\right)\varphi^\frac{1+\theta}{2}\Bigg).
    \end{aligned}
    $$
    Choosing $\theta > \frac{1}{q}$ and $\gamma \geq \frac{1}{q-1}$ gives
    $$\max_{B(0,1)} \varphi w_\lambda = \varphi(x_1)w_\lambda(x_1) \leq C(1 + f_3(\mu)^\frac{2q}{q-\widetilde{q}} + \Lambda_{q'}(\mu)).$$
    In particular, $w_\lambda(0) = \varphi(0)w_\lambda(0) \leq C(1 + f_3(\mu)^\frac{2q}{q-\widetilde{q}} + \Lambda_{q'}(\mu))$ and so
    $$|D_xu_\lambda| \leq C(1 + f_3(\mu)^\frac{q}{q-\widetilde{q}} + \Lambda_{q'}(\mu)^\frac{1}{2})r^{-\gamma}.$$
\end{proof}

Finally, in order to use some known results for the Fokker-Planck equation, we need to investigate the asymptotic behavior of the value function and its derivatives as $d(x) \to 0$. For this, we adapt the arguments used to prove \cite[Theorem II.3]{lasry1989nonlinear} and \cite[Proposition 3.2]{porretta2020mean}, respectively.

\begin{lemma}
    Assume A\ref{A: Hamiltonian}-\ref{A: F} and A\ref{A: Lower bound for L}-\ref{A: Bound for L} hold. Now let $u$ be a solution of \eqref{Eq: Ergodic HJB}.
    Then
    \begin{equation} \label{Eq: Second-Order Expansion of u}
        u = 
        \begin{cases}
            (q-1)^{2-q'}(2-q)^{-1}f_1(\mu)^{1-q'}d(x)^{2-q'} + O(d(x)^{3-q'}), &1 < q < \frac{3}{2} \\
            (q-1)^{2-q'}(2-q)^{-1}f_1(\mu)^{1-q'}d(x)^{2-q'} + O(|\ln{d(x)}|), &q = \frac{3}{2} \\
            (q-1)^{2-q'}(2-q)^{-1}f_1(\mu)^{1-q'}d(x)^{2-q'} + O(1), &\frac{3}{2} < q < 2 \\
            -\frac{1}{f_1(\mu)}\ln{d(x)} + O(1), &q = 2 \\
        \end{cases}
    \end{equation}
\end{lemma}

\begin{proof}
    As in the proof of \cite[Theorem II.3]{lasry1989nonlinear}, it suffices to find appropriate sub- and super-solutions to
    \begin{equation}
        \mathcal{L}[v] = u + F(\mu,x) - \rho
    \end{equation}
    where $\mathcal{L}[v] \coloneqq -\Delta v + H(D_xv,\mu) + v$.
    We claim that for sufficiently large constants $A_1$ and $A_2$ (depending on $A_1$ and $q$),
    \begin{equation}
        w_\varepsilon^+ =
        \begin{cases}
            (q-1)^{2-q'}(2-q)^{-1}f_1(\mu)^{1-q'}d^{2-q'} + A_1d^{3-q'} + A_2, &1 < q < \frac{3}{2} \\
            (q-1)^{2-q'}(2-q)^{-1}f_1(\mu)^{1-q'}d^{-1} - A_1\ln{d} + A_2, &q = \frac{3}{2} \\
            (q-1)^{2-q'}(2-q)^{-1}f_1(\mu)^{1-q'}d^{2-q'} - A_1d^{3-q'} + A_2, &\frac{3}{2} < q < 2 \\
            -\frac{1}{f_1(\mu)}\ln{d} - A_1d^{3-q'} + A_2, &q = 2
        \end{cases}
    \end{equation}
    is a super-solution and
    \begin{equation}
        w_\varepsilon^- =
        \begin{cases}
            (q-1)^{2-q'}(2-q)^{-1}f_1(\mu)^{1-q'}d^{2-q'} - A_1d^{3-q'} - A_2, &1 < q < \frac{3}{2} \\
            (q-1)^{2-q'}(2-q)^{-1}f_1(\mu)^{1-q'}d^{-1} + A_1\ln{d} - A_2, &q = \frac{3}{2} \\
            (q-1)^{2-q'}(2-q)^{-1}f_1(\mu)^{1-q'}d^{2-q'} + A_1d^{3-q'} - A_2, &\frac{3}{2} < q < 2 \\
            -\frac{1}{f_1(\mu)}\ln{d} + A_1d^{3-q'} - A_2, &q = 2
        \end{cases}
    \end{equation}
    is a sub-solution. Since $q' < 3$ for $q > \frac{3}{2}$, this would be sufficient to prove the theorem.

    First, we recall that the map $x \mapsto |x|^q$ is convex and hence for $a,b \in \R^n$, we have
    \begin{equation} \label{Eq: Inequalities by Convexity}
        |a|^q + q|a|^{q-2}a \cdot b \leq |a+b|^q \leq |a|^q + q|a+b|^{q-2}(a+b) \cdot b.
    \end{equation}
    We will only prove the first case (i.e.~$1 < q < \frac{3}{2}$) as the others follow by very similar arguments.
    Note that in $\Gamma_\delta$ for $\delta > 0$ small enough,
    $$D_xw_\varepsilon^+ = -\left[(q-1)^{1-q'}f_1(\mu)^{1-q'}d^{1-q'} + A_1(q'-3)d^{2-q'}\right]D_xd$$
    and
    $$
    \begin{aligned}
        \Delta w_\varepsilon^+ &= (q-1)^{-q'}f_1(\mu)^{1-q'}d^{-q'} + A_1(q'-3)(q'-2)d^{1-q'} \\
        &\qquad - \left[(q-1)^{1-q'}f_1(\mu)^{1-q'}d^{1-q'} + A_1(q'-3)d^{2-q'}\right]\Delta d,
    \end{aligned}
    $$
    where we use that $|D_xd| = 1$ in $\Gamma_\delta$. Thus, A\ref{A: Bounds for H} gives
    $$
    \begin{aligned}
        \mathcal{L}[w_\varepsilon^+] &\geq -(q-1)^{-q'}f_1(\mu)^{1-q'}d^{-q'} - A_1(q'-3)(q'-2)d^{1-q'} \\
        &\qquad + \left[(q-1)^{1-q'}f_1(\mu)^{1-q'}d^{1-q'} + A_1(q'-3)d^{2-q'}\right]\Delta d \\
        &\qquad + f_1(\mu)\left|(q-1)^{1-q'}f_1(\mu)^{1-q'}d^{1-q'} + A_1(q'-3)d^{2-q'}\right|^q \\
        &\qquad -\left(\left|(q-1)^{1-q'}f_1(\mu)^{1-q'}d^{1-q'} + A_1(q'-3)d^{2-q'}\right|^{\widetilde{q}} + 1\right)f_3(\mu) \\
        &\qquad + (q-1)^{2-q'}(2-q)^{-1}f_1(\mu)^{1-q'}d^{2-q'} + A_1d^{3-q'} + A_2.
    \end{aligned}
    $$
    Recalling that $q(1-q') = -q'$ and $(1-q')(q-1) = -1$, \eqref{Eq: Inequalities by Convexity} gives
    $$
    \begin{aligned}
        &\left|(q-1)^{1-q'}f_1(\mu)^{1-q'}d^{1-q'} + A_1(q'-3)d^{2-q'}\right|^q \\
        &\geq (q-1)^{q(1-q')}f_1(\mu)^{q(1-q')}d^{q(1-q')} + A_1q(q'-3)(q-1)^{(1-q')(q-1)}f_1(\mu)^{(1-q')(q-1)}d^{(1-q')(q-1) + 2 - q'} \\
        &= (q-1)^{-q'}f_1(\mu)^{-q'}d^{-q'} + A_1q'(q'-3)f_1(\mu)^{-1}d^{1-q'}
    \end{aligned}
    $$
    Using Young's inequality and the fact that $u + F \leq C(d^{2-q'} + 1)$, for all $\varepsilon > 0$, we get
    $$
    \begin{aligned}
        \mathcal{L}[w_\varepsilon^+] &\geq \left[A_1(q' - (q'-2))(q'-3) + (q-1)^{1-q'}f_1(\mu)^{1-q'}\Delta d - \varepsilon\right]d^{1-q'}  - C_\varepsilon + A_2 \\
        &= \left[2A_1(q'-3) + (q-1)^{1-q'}f_1(\mu)^{1-q'}\Delta d - \varepsilon\right]d^{1-q'} - C_\varepsilon + A_2 \\
        &\geq \left[2A_1(q'-3) - (q-1)^{1-q'}f_1(\mu)^{1-q'}\|\Delta d\|_\infty - 2\varepsilon\right]d^{1-q'} + u - \rho - \widetilde{C}_\varepsilon + A_2 \\
        &\geq u + F(\mu,x) - \rho
    \end{aligned}
    $$
    provided $A_1 \geq \frac{1}{2}(q'-3)^{-1}(q-1)^{1-q'}f_1(\mu)^{1-q'}\|\Delta d\|_\infty + \varepsilon(q'-3)^{-1}$ and $A_2 \geq \widetilde{C}_\varepsilon$. Similarly, we get
    $$
    \begin{aligned}
        \mathcal{L}[w_\varepsilon^-] &\leq -(q-1)^{-q'}f_1(\mu)^{1-q'}d^{-q'} + A_1(q'-3)(q'-2)d^{1-q'} \\
        &\qquad + \left[(q-1)^{1-q'}f_1(\mu)^{1-q'}d^{1-q'} - A_1(q'-3)d^{2-q'}\right]\Delta d \\
        &\qquad + f_1(\mu)\left|(q-1)^{1-q'}f_1(\mu)^{1-q'}d^{1-q'} + A_1(q'-3)d^{2-q'}\right|^q \\
        &\qquad +\left(\left|(q-1)^{1-q'}f_1(\mu)^{1-q'}d^{1-q'} + A_1(q'-3)d^{2-q'}\right|^{\widetilde{q}} + 1\right)f_3(\mu) \\
        &\qquad + (q-1)^{2-q'}(2-q)^{-1}f_1(\mu)^{1-q'}d^{2-q'} - A_1d^{3-q'} - A_2
    \end{aligned}
    $$
    and
    $$
    \begin{aligned}
        &\left|(q-1)^{1-q'}f_1(\mu)^{1-q'}d^{1-q'} - A_1(q'-3)d^{2-q'}\right|^q \\
        &\leq (q-1)^{q(1-q')}f_1(\mu)^{q(1-q')}d^{q(1-q')} \\
        &\qquad- A_1q(q'-3)d^{2-q'}\left|(q-1)^{1-q'}f_1(\mu)^{1-q'}d^{1-q'} - A_1(q'-3)d^{2-q'}\right|^{q-1} \\
        &\leq (q-1)^{-q'}f_1(\mu)^{-q'}d^{-q'} - A_1q'(q'-3)f_1(\mu)^{-1}d^{1-q'} + A_1^q q(q'-3)^q d^{q-q'}
    \end{aligned}
    $$
    in $\Gamma_\delta$ for $\delta$ sufficiently small, where the last inequality follows from the fact that $|a+b|^{q-1} \geq |a|^{q-1} - |b|^{q-1}$ for $a,b \in \R^n$. Thus, since $u + F \geq C$, for all $\varepsilon > 0$, we get
    $$
    \begin{aligned}
        \mathcal{L}[w_\varepsilon^-] &\leq \left[-2A_1(q'-3) + (q-1)^{1-q'}f_1(\mu)^{1-q'}\Delta d + \varepsilon\right]d^{1-q'} + C_\varepsilon - A_2 \\
        &\leq \left[-2A_1(q'-3) + (q-1)^{1-q'}f_1(\mu)^{1-q'}\|\Delta d\|_\infty + \varepsilon\right]d^{1-q'} + u - \rho + \widetilde{C}_\varepsilon - A_2 \\
        &\leq u + F(\mu,x) - \rho
    \end{aligned}
    $$
    provided $A_1 \geq \frac{1}{2}(q'-3)^{-1}(q-1)^{1-q'}f_1(\mu)^{1-q'}\|\Delta d\|_\infty + \frac{\varepsilon}{2}(q'-3)^{-1}$ and $A_2 \geq \widetilde{C}_\varepsilon$.
\end{proof}

\begin{lemma} \label{Lem: Asymptotic Behavior of Derivatives}
    Assume A\ref{A: F}-\ref{A: Regularity of G} hold and let $(u,\rho)$ be a solution of \eqref{Eq: Ergodic HJB}. Then for $1 < q < 2$, we have
    \begin{equation} \label{Eq: Second-Order Expansion of Du}
        D_xu(x) = (f_1(\mu)(q-1)d(x))^{1-q'}[\n(x) + O(\omega_q(d(x)))]
    \end{equation}
    as $d(x) \to 0$, where
    $$\omega_q(\delta) \coloneqq
    \begin{cases}
        \delta, &1 < q < \frac{3}{2} \\
        \delta|\ln{\delta}|, & q = \frac{3}{2} \\
        \delta^{q' - 2}, &\frac{3}{2} < q < 2
    \end{cases}
    $$
    and if A\ref{A: q=2} holds, then we have
    $$D_xu(x) = (f_1(\mu)(q-1)d(x))^{-1}[\n(x) + O(d(x))]$$
    as $d(x) \to 0$ for $q=2$.
    Furthermore, we have
    $$\lim_{x \to x_0 \in \B} d(x)^{q'} D_{xx}^2u(x) = f_1(\mu)^{1-q'}(q-1)^{-q'}\n(x_0) \otimes \n(x_0),$$
    and for $1 < q < 2$, we have
    \begin{equation}
        D_{xx}^2u(x) = f_1(\mu)^{1-q'}(q-1)^{-q'}d(x)^{-q'}[\n(x) \otimes \n(x) + O(\omega_q(d(x)))].
    \end{equation}
\end{lemma}

\begin{proof}
    Let $\delta_0 > 0$ be sufficiently small so that $d(x) = d(x,\B)$ in $\Gamma_{\delta_0}$. Next, we fix $x_0 \in \B$ and consider a new orthonormal basis $\{v_1,\dots,v_n\}$ for $\R^n$ with $v_1 = -\n(x_0)$. We will use $(y_1,\dots,y_n)$  to denote the related system of coordinates centered at $x_0$. In these coordinates, letting $0 < \delta < \delta_0$, $0 < \zeta < \frac{1}{2}$, and $O_{\delta_0} = (\delta_0,0,\dots,0)$, we define
    $$D_\delta \coloneqq B_{\delta^{1-\zeta}} \cap  B_{\delta_0}(O_{\delta_0}).$$
    Since $\zeta > 0$, we have $\frac{1}{\delta}D_\delta \to \R^n_+ \coloneqq \{y \in \R^n : y_1 >0\}$ as $\delta \to 0$. Making another change of variables, we define $\xi = \frac{y}{\delta}$ and
    $$v_\delta(\xi) =
    \begin{cases}
        (q-1)^{q'-1}\delta^{q'-2}u(\delta\xi), &q < 2 \\
        u(\delta\xi) + \frac{1}{f_1(\mu)}\ln{\delta}, &q = 2
    \end{cases}$$
    By \eqref{Eq: Second-Order Expansion of u}, we get that $v_\delta$ is locally bounded for $\xi_1 > 0$, uniformly in $\delta$. Moreover, $v_\delta$ satisfies the equation
    \begin{equation} \label{Eq: v-delta}
        -\Delta v_\delta(\xi) + (q-1)^{q'-1}\delta^{q'}H((q-1)^{1-q'}\delta^{1-q'}D_\xi v_\delta,\mu) = (q-1)^{q'-1}\delta^{q'}(F - \rho)
    \end{equation}
    for $\xi \in \frac{1}{\delta}D_\delta$, and $\delta^{q'}H((q-1)^{1-q'}\delta^{1-q'}D_\xi v_\delta,\mu)$ is locally bounded by Theorem \ref{Thm: Gradient Estimate (Given mu)}. By elliptic regularity, we get that $v_\delta$ is locally bounded in $C^{2+\beta}$. Using relative compactness and a diagonal argument, there exists a function $v \in C^2_{loc}$ and a subsequence $v_{\delta_k}$ converging to $v$ locally in $C^2$ for all. Passing to the limit, we have
    \begin{equation} \label{Eq: Limit of v-delta}
        -\Delta v(\xi) + f_1(\mu)(q-1)^{-1}|D_\xi v(\xi)|^q = 0.
    \end{equation}

    For $q < 2$, we use \eqref{Eq: Boundary Behavior of u} to obtain
    $$\lim_{\delta \to 0} (\delta\xi_1)^{q'-2}u(\delta\xi) = (q-1)^{2-q'}(2-q)^{-1}f_1(\mu)^{-\frac{1}{q-1}}$$
    which implies that
    $$\lim_{\delta \to 0} v_{\delta}(\xi) = \frac{q-1}{2-q}f_1(\mu)^{-\frac{1}{q-1}}\xi_1^{2-q'}.$$
    For $q = 2$, we recall that $u(x) + \frac{1}{f_1(\mu)}\ln{\delta}$ is bounded. Thus, $w = e^{-f_1(\mu)v}$ is positive and harmonic on $\{\xi \in \R^n : \xi_1 > 0\}$ with $w \leq C\xi_1$ for some $C > 0$. Therefore, $w = \lambda\xi_1$ for some $\lambda > 0$, and hence
    $$\lim_{k \to \infty} v_{\delta_k}(\xi) = -\frac{1}{f_1(\mu)}\ln{\xi_1} - \widetilde{\lambda}.$$
    By uniqueness, we get the convergence of the sequences $D_xv_\delta, D_{xx}^2v_\delta$. In particular,
    $$\lim_{\delta \to 0} D_\xi v_\delta = D_\xi v = -f_1(\mu)^{-\frac{1}{q-1}}\xi_1^{1-q'}\vec{e}_1$$
    where $\vec{e}_1 = (1,0,\dots,0)$, and
    $$\lim_{\delta \to 0} D_{\xi\xi}^2 v_\delta = D_{\xi\xi}^2 v =
    (q-1)^{-1}f_1(\mu)^{-\frac{1}{q-1}}\xi_1^{-q'}\begin{pmatrix}
        1 & 0 & \cdots & 0 \\
        0 & \ddots & & \vdots \\
        \vdots & & & \\
        0 & \cdots & & 0
    \end{pmatrix}.$$
    Since $D_\xi u(\delta\xi) = (q-1)^{1-q'}\delta^{1-q'}D_\xi v_\delta(\xi)$ and $D_{\xi\xi}^2 u(\delta\xi) = (q-1)^{1-q'}\delta^{-q'}D_{\xi\xi}^2 v_\delta(\xi)$, it follows that $\underset{y_1 \to 0}{\lim} \partial_{y_i}u(y) = 0$ for $i \neq 1$, $\underset{y_1 \to 0}{\lim} \partial_{y_iy_j}^2u(y) = 0$ for $(i,j) \neq (1,1)$. Moreover, choosing $\xi = \vec{e}_1$ gives
    $$\lim_{\delta \to 0} \delta^{q'-1} D_{y_1}u(\delta,0,\dots,0) = -(q-1)^{1-q'}f_1(\mu)^{-\frac{1}{q-1}}$$
    and
    $$\lim_{\delta \to 0} \delta^{q'} D_{y_1y_1}^2u(\delta,0,\dots,0) = (q-1)^{-q'}f_1(\mu)^{-\frac{1}{q-1}}.$$
    As $\partial_{y_1}u(\delta,0,\dots,0) = -\partial_{\n(x_0)}u(x_0 - \delta\n(x_0))$ and $\partial_{y_1y_1}^2u(\delta,0,\dots,0) = \partial_{\n(x_0)\n(x_0)}^2 u(x_0-\delta\n(x_0))$, we now have the ``first-order expansions" for $D_xu$ and $D_{xx}^2u$.
    
    What remains is to prove the ``second-order expansions". First, we note that if $q = 2$ and A\ref{A: q=2} holds, then $\widetilde{u} \coloneq \psi(\mu)(u - x \cdot \varphi(\mu))$ satisfies
    $$-\Delta \widetilde{u} + |D_x\widetilde{u}|^2 + \psi(\mu)\rho = \psi(\mu)F(\mu,x) -\psi(\mu)V(\mu)$$
    and so
    $$D_xu = \psi(\mu)^{-1}D_x\widetilde{u} + \varphi(\mu) = \frac{1}{\psi(\mu)d(x)}[\n(x) + O(1)]$$
    as $d(x) \to 0$ by \cite[Equation (3.5)]{porretta2023ergodic}.
    
    Now assume $q < 2$. We set $\widetilde{v}(\xi) = v(\xi) - \delta^{q'-1}(q-1)^{q'-1}\xi \cdot f_2(\mu)$ and use \eqref{Eq: v-delta},\eqref{Eq: Limit of v-delta} to obtain
    $$-\Delta(\widetilde{v}-v_\delta) + D_\xi (\widetilde{v} - v_\delta) \cdot V_\delta = (q-1)^{q'-1}\delta^{q'}(\rho - F) + g_\delta$$
    where
    $$V_\delta = \delta\int_0^1D_pH((q-1)^{1-q'}\delta^{1-q'}(sD_\xi v_\delta + (1-s)D_\xi \widetilde{v}),\mu) ds$$
    and
    $$
    \begin{aligned}
        g_\delta &= (q-1)^{q'-1}\delta^{q'}H\left((q-1)^{1-q'}\delta^{1-q'}D_\xi \widetilde{v},\mu\right) - f_1(\mu)(q-1)^{-1}\left|D_\xi \widetilde{v} + \delta^{q'-1}(q-1)^{q'-1}f_2(\mu)\right|^q \\
        &= (q-1)^{q'-1}\delta^{q'}G\left((q-1)^{1-q'}\delta^{1-q'}D_\xi \widetilde{v},\mu\right)
    \end{aligned}
    $$
    by A\ref{A: Bounds for H}. Since $V_\delta$ is locally bounded, for sufficiently small $\delta > 0$, we can apply \cite[Theorem 6.2]{gilbarg1977elliptic} to the set
    $$A \coloneqq \left\{\xi \in \R^n_+ : |\xi - \vec{e}_1| < \frac{1}{2}\right\},$$
    which gives
    $$|D_\xi \widetilde{v}(\vec{e}_1) - D_\xi v_\delta(\vec{e}_1)| + |D_{\xi\xi}^2 \widetilde{v}(\vec{e}_1) - D_{\xi\xi}^2 v_\delta(\vec{e}_1)| \leq C(\|\widetilde{v} - v_\delta\|_\infty + \|g_\delta\|_{C^\beta(A)} + \delta^{q'}).$$
    Note that A\ref{A: Bounds for H} and A\ref{A: Regularity of G} give
    $$\left|G\left((q-1)^{1-q'}\delta^{1-q'}D_\xi \widetilde{v},\mu\right)\right| \leq Cf_3(\mu)(1 + \delta^{(1-q')\widetilde{q}})$$
    and
    $$\left|G\left((q-1)^{1-q'}\delta^{1-q'}D_\xi \widetilde{v}(\xi^1),\mu\right) - G\left((q-1)^{1-q'}\delta^{1-q'}D_\xi \widetilde{v}(\xi^2),\mu\right)\right| \leq Cf_3(\mu)\delta^{(1-2q')\widetilde{\alpha}}|\xi^1 - \xi^2|^{\widetilde{\alpha}}.$$
    
    By \eqref{Eq: Second-Order Expansion of u}, we have
    $$\|v - v_\delta\|_\infty = O(\omega_q(\delta))$$
    and so
    $$|D_\xi \widetilde{v}(\vec{e}_1) - D_\xi v_\delta(\vec{e}_1)| +
    |D_{\xi\xi}^2 \widetilde{v}(\vec{e}_1) - D_{\xi\xi}^2 v_\delta(\vec{e}_1)| = O(\omega_q(\delta)).$$
    Since
    $$|D_\xi \widetilde{v}(\vec{e}_1) - D_\xi v(\vec{e}_1)| +
    |D_{\xi\xi}^2 \widetilde{v}(\vec{e}_1) - D_{\xi\xi}^2 v(\vec{e}_1)| = O(\delta^{q'-1}),$$
    this completes the proof.
\end{proof}

\section{Fokker-Planck Equation}
\label{Sec: FP}

In this section, we recall results from \cite{porretta2023ergodic} on the well-posedness of the Fokker-Planck equation and the regularity of solutions. As in Section \ref{Sec: HJ}, there is no loss of generality in assuming $\sigma = 1$.

\begin{definition} \label{Def: Solution to FPK}
    Given $b \in L^\infty_{loc}(\Omega;\R^n)$, we say $m \in \mathcal{P}(\Omega)$ is a \textit{weak solution} of
    \begin{equation} \label{Eq: FPK}
        \Delta m + \nabla \cdot (mb) = 0, \qquad x \in \Omega
    \end{equation}
    if for all $\varphi \in L^\infty(\Omega)$ such that there is a bounded continuous function $\eta$ for which $-\Delta\varphi + b \cdot D_x\varphi = \eta$ in the sense of distributions, we have
    $$\int_\Omega [-\Delta\varphi + b \cdot D_x\varphi]dm \coloneqq \int_\Omega \eta \ dm = 0.$$
\end{definition}
We remark that in Definition \ref{Def: Solution to FPK}, it is sufficient to take $\eta$ in a set that is dense in the space of bounded continuous functions, e.g.~we can take $\eta \in W^{1,\infty}(\Omega)$ (cf.~the proof of Theorem \ref{Thm: Existence} below).

\begin{theorem} \label{Thm: Well-Posedness of FPK}
    Suppose $b \in C^0(\Omega;\R^n) \cap W^{1,\infty}_{loc}(\Gamma_{\delta_0};\R^n)$ for some $\delta_0 > 0$ and that either
    $$
    \begin{cases}
        \Delta d - b \cdot D_xd \geq \frac{1}{d} - Cd \qquad \text{ for } x \in \Gamma_{\delta_0} \text{ for some } C > 0 \\
        \operatorname{Jac} b \geq -Cd^{\gamma_0-2}I  \qquad \text{ for } x \in \Gamma_{\delta_0} \text{ for some } C > 0, \gamma_0 > 0
    \end{cases}
    $$
    or
    $$
    \begin{cases}
        \Delta d - b \cdot D_xd \geq \frac{\beta_0}{d}(1 + o(1)) \qquad \text{ for } x \in \Gamma_{\delta_0} \text{ for some } \beta_0 > 1 \\
        \operatorname{Jac} b \geq -Cd^{-2}\kappa I  \qquad \text{ for } x \in \Gamma_{\delta_0} > 0 \text{ for some } C > 0
    \end{cases}
    $$
    where $\kappa(x) \to 0$ as $x \to \B$. Then there is a unique weak solution of \eqref{Eq: FPK}, which is absolutely continuous with $L^1$ density.
\end{theorem}

\begin{theorem} \label{Thm: Regularity of m}
    Assume the hypotheses of Theorem \ref{Thm: Well-Posedness of FPK} hold. Now assume there exist $\gamma > 1$ and $\delta_0, \theta > 0$ such that for all $x \in \Gamma_{\delta_0}$,
    $$b(x) = \frac{\gamma}{d(x)}[\n(x) + O(d(x)^\theta)] \qquad \text{and} \qquad \nabla \cdot b = \frac{\gamma}{d(x)^2}[1 + O(d(x)^\theta)].$$
    Let $m$ be the weak solution of \eqref{Eq: FPK}. Then $m \in C^{1+\beta}(\overline{\Omega})$ for some $\beta > 0$ and there exist $A_1,A_2 > 0$ such that
    $$A_1d^\gamma \leq m \leq A_2d^\gamma.$$
\end{theorem}

\begin{remark}
    We observe that if we assume A\ref{A: Hamiltonian}-\ref{A: Regularity of G} and either A\ref{A: Asymptotics of DpH} (for $1 < q < 2)$ or A\ref{A: q=2} (for $q=2$), and if $u$ is a solution to the Hamilton-Jacobi equation, then the conditions of Theorems \ref{Thm: Well-Posedness of FPK} and \ref{Thm: Regularity of m} are satisfied for $b = D_pH(D_xu,\mu)$. In the case $q < 2$, this follows by combining Lemma \ref{Lem: Asymptotic Behavior of Derivatives} with A\ref{A: Asymptotics of DpH}. For $q = 2$, we apply gradient blow-up to find $\delta_0 > 0$ such that $|D_xu| > |\varphi|$ in $\Gamma_{\delta_0}$, which gives $b \in W^{1,\infty}_{loc}(\Gamma_{\delta_0})$. The other conditions follow from the fact that
    $$\operatorname{Jac}(D_pH(D_xu,\mu)) = 2\psi(\mu)D_{xx}^2u,$$
    which implies
    $$\nabla \cdot (D_pH(D_xu,\mu)) = 2\psi(\mu)\Delta u = 2\psi(\mu)^2|D_xu + \varphi(\mu)|^2 + 2\psi(\mu)(V(\mu) + \rho) = \frac{2}{d^2}(1 + O(d))$$
    and
    $$\langle(\operatorname{Jac}(D_pH(v,\mu)))\xi,\xi\rangle = 2\psi(\mu)\langle(D_{xx}^2u)\xi,\xi\rangle = \frac{q'}{d^2}\left[|\langle\xi,\n(x)\rangle|^2 + o(1)\right] \geq - Cd(x)^2\kappa(x).$$
    where $\kappa(x) \to 0$ as $d(x) \to 0$.
\end{remark}

\section{A Priori Estimates}
\label{Sec: A Priori Estimates}

In this section, we obtain a priori estimates for solutions to \eqref{Eq: Ergodic MFG}. Our approach is inspired by the one found in \cite{kobeissi2022mean}, in which a priori estimates for solutions are obtained using a comparison with a fixed control. This allows us to obtain a priori estimates for solutions without requiring any smallness conditions. However, since the control $a_0 = 0$ is not admissible, we will need to use another control for comparison. For this purpose, we will choose the control $b_0(\cdot)$ given by
\begin{equation} \label{Eq: b0}
    \begin{cases}
        b_0(\cdot) \coloneqq -q|D_xv(\cdot)|^{q-2}D_xv(\cdot) \\
        -\sigma\Delta v + |D_xv|^q + \widetilde{\rho} = 0
    \end{cases}
\end{equation}

We begin with the following adaptation of \cite[Theorem VII.3]{lasry1989nonlinear}.

\begin{lemma} \label{Lem: Stochastic Interpretation}
    Assume A\ref{A: F}-\ref{A: Bound for L} hold and let $(u,\rho,m,\mu)$ be the unique solution of \eqref{Eq: Ergodic MFG}.
    Consider the controlled dynamics
    $$dX_t = a(X_t)dt + \sqrt{2\sigma}dB_t, \qquad 0 \leq t < \tau_x, \qquad X_0 = x \in \Omega,$$
    where $\tau_x \coloneqq \inf\{t \geq 0 : X_t \notin \Omega\}$, and $\mathcal{A}$ is the set of admissible controls, i.e.~the set of all measurable $a(\cdot)$ such that $P(\tau_x < \infty) = 0$ for all $x \in \Omega$.
    For each $a \in \mathcal{A}$ let $\theta_a$ be a stopping time bounded by some constant independent of $a$.
    Then we have
    \begin{equation} \label{Eq: Stochastic Interpretation of rho}
        \rho = \lim_{T \to \infty} \frac{1}{T} \inf_{a \in \mathcal{A}} E \int_0^T \left(L(a(X_t),\mu) + F(\mu,X_t)\right) dt,
    \end{equation}
    \begin{equation} \label{Eq: Stochastic Interpretation of u}
        u(x) = \inf_{a \in \mathcal{A}} E \left[\int_0^{\theta_a} \left(L(a(X_t),\mu) + F(\mu,X_t)\right) dt + u(X_{\theta_a}) - \theta_a\rho\right].
    \end{equation}
     Furthermore, $b_0(\cdot), -D_pH(D_xu(\cdot),\mu) \in \mathcal{A}$, where $b_0(\cdot)$ is given by \eqref{Eq: b0}, and the infimums in \eqref{Eq: Stochastic Interpretation of rho}, \eqref{Eq: Stochastic Interpretation of u} are attained by $a = -D_pH(D_xu,\mu)$.
\end{lemma}

\begin{proof}
    First, we show that $a = -D_pH(D_xu,\mu)$ is an admissible control. To this end, define $X_t$ by
    $$dX_t = a(X_t)dt + \sqrt{2\sigma}dB_t \qquad X_0 = x.$$
    Then for $\delta > 0$, It\^{o}'s formula gives
    \begin{equation} \label{Eq: Ito representation of u}
        \begin{aligned}
            E[u(X_{\theta_a \land \tau_x^\delta})] &= u(x) - E\left[\int_0^{\theta_a \land \tau_x^\delta} \left(L(a(X_t),\mu) + F(\mu,X_t)\right) dt - (\theta_a \land \tau_x^\delta)\rho\right] \\
            &\leq u(x) + C\theta_a
        \end{aligned}
    \end{equation}
    where $\tau_x^\delta \coloneqq \inf\{t \geq 0 : X_t \notin \Omega_\delta\}$. Since $u$ is bounded below, this implies
    $$\left(\inf_{\B_\delta} u\right)P\left(\tau_x^\delta \leq \theta_a\right) \leq u(x) + C(1 + \theta_a).$$
    By \eqref{Eq: Boundary Behavior 3}, this implies that $a \in \mathcal{A}$. By a similar argument, we deduce that $b_0(\cdot) \in \mathcal{A}$.

    Letting $\delta \to 0$ in \eqref{Eq: Ito representation of u} gives
    $$u(x) = E\left[\int_0^{\theta_a} \left(L(a(X_t),\mu) + F(\mu,X_t)\right) dt\right] - \theta_a\rho + \lim_{\delta \to 0} E\left[u(X_{\theta_a \land \tau_x^\delta})\right]$$
    and
    $$\lim_{\delta \to 0} E\left[u(X_{\theta_a \land \tau_x^\delta})\right] \geq \lim_{\delta \to 0} E\left[\left(u(X_{\theta_a}) + C\right)\chi_{\theta_a \leq \tau_x^\delta}\right] - C$$
    where $C \geq -\underset{\Omega}{\inf}u$, and the last expression increases to $E\left[u(X_{\theta_a})\right]$. Thus, we have
    $$u(x) \geq E\left[\int_0^{\theta_a} \left(L(a(X_t),\mu) + F(\mu,X_t)\right) dt + u(X_{\theta_a})\right] - \theta_a\rho,$$
    and letting $\theta_a = T$ and then $T \to \infty$, we have
    $$\rho \geq \limsup_{T \to \infty} \frac{1}{T}E\left[\int_0^T \left(L(a(X_t),\mu) + F(\mu,X_t)\right) dt\right].$$
    
    What remains is to show the complementary inequalities. To this end, define $(u^\delta,\rho^\delta)$ to be the solution of \eqref{Eq: Ergodic HJB} with $u^\delta(x_0) = 0$ for $\mu = (I,-D_pH(D_xu,\mu)) \# m$, with $\Omega$ replaced by $\Omega^\delta$. As in the proof of Theorem \ref{Thm: HJB has a solution}, using the a priori bounds on $\rho^\delta$ and local $W^{2,\infty}$ estimates on $u^\delta$, we conclude that $u^\delta \to u$ as $\delta \to 0$ uniformly on compact subsets of $\Omega$, and that $\rho^\delta \to \rho$.
    By It\^{o}'s formula, we get that for every $x \in \Omega$ and $\alpha \in \mathcal{A}$,
    $$
    \begin{aligned}
        u^\delta(x) &= E\left[\int_0^{\theta_\alpha} (-\alpha(X_t) \cdot D_xu^\delta(X_t) - H(D_xu^\delta(X_t),\mu) + F(\mu,X_t)) dt + u^\delta(X_{\theta_\alpha}) - \theta_\alpha\rho^\delta\right] \\
        &\leq E\left[\int_0^{\theta_\alpha} \left(L(\alpha(X_t),\mu) + F(\mu,X_t)\right) dt + u^\delta(X_{\theta_\alpha}) - \theta_\alpha\rho^\delta\right].
    \end{aligned}
    $$
    As $\delta \to 0$, since $u^\delta \to u$ and $\rho^\delta \to \rho$, we deduce that
    $$u(x) \leq \inf_{a \in \mathcal{A}} E\left[\int_0^{\theta_a} \left(L(a(X_t),\mu) + F(\mu,X_t)\right) dt + u(X_{\theta_a}) - \theta_a\rho\right].$$
    As before, taking $\theta_a = T$, we get
    $$\rho^\delta \leq \inf_{a \in \mathcal{A}} \frac{1}{T} E\left[\int_0^T \left(L(a(X_t),\mu) + F(\mu,X_t)\right) dt\right] + \frac{2}{T}\sup_\Omega |u^\delta|$$
    and so
    $$\rho^\delta \leq \liminf_{T \to \infty} \inf_{a \in \mathcal{A}} \frac{1}{T} E\left[\int_0^T \left(L(a(X_t),\mu) + F(\mu,X_t)\right) dt\right].$$
    Letting $\delta \to 0$ completes the proof that $a = -D_pH(D_xu,\mu)$ is a minimizer for \eqref{Eq: Stochastic Interpretation of rho}, \eqref{Eq: Stochastic Interpretation of u}.
\end{proof}

To prove our a priori estimate, we use the following adaptation of \cite[Lemma 5.2]{porretta2023ergodic}, which essentially allows us to justify integrating by parts. The proof is nearly identical and is therefore omitted.

\begin{lemma} \label{Lem: u as a Test Function}
    Let $m$ be a weak solution to
    $$\sigma\Delta m + \nabla \cdot (mb) = 0$$
    for some $b \in C^0(\Omega;\R^n) \cap W^{1,\infty}_{loc}(\Gamma_{\delta_0};\R^n)$ satisfying 
    $$b(x) = \frac{\sigma q'}{d(x)}[\nu(x) + O(d(x)^\theta)]$$
    for some $\delta_0,\theta > 0$. If $(u,\rho)$ is a solution to 
    $$-\sigma\Delta u + H(D_xu,\mu) + \rho = F(\mu,x)$$
    for some $\mu \in \mathcal{P}_{q'}(\overline{\Omega} \times \R^n)$, then
    $$\int_\Omega (-\sigma\Delta u + b \cdot D_xu)dm = 0.$$
\end{lemma}

\begin{remark}
    Note that by Lemmas \ref{Lem: Stochastic Interpretation} and \ref{Lem: u as a Test Function}, if A\ref{A: Hamiltonian}-\ref{A: Regularity of G} and either A\ref{A: Asymptotics of DpH} or A\ref{A: q=2} hold, then we get
    $$
    \begin{aligned}
        \int_{\Omega \times \R^n} (L(\alpha,\mu) + F(\mu,x)) d\mu(x,\alpha) &= \int_{\Omega} (L(-D_pH(D_xu,\mu),\mu) + F(\mu,x)) dm \\
        &= \int_{\Omega} (D_pH(D_xu,\mu) \cdot D_xu - H(D_xu,\mu) + F(\mu,x)) dm \\
        &= \rho \\
        &= \lim_{T \to \infty} \frac{1}{T} E \int_0^T (L(-D_pH(D_xu(X_t),\mu),\mu) + F(\mu,X_t)) dt
    \end{aligned}
    $$
    where $dX_t = -D_pH(D_xu(X_t),\mu)dt + \sqrt{2\sigma}dB_t$.
\end{remark}

\begin{theorem} \label{Thm: A Priori Estimate 1}
    Assume A\ref{A: Hamiltonian}-\ref{A: Regularity of G} and either A\ref{A: Asymptotics of DpH} or A\ref{A: q=2} hold. Then there exists some $C > 0$ so that
    $$\Lambda_{q'}(\mu) \leq C$$
    for every solution $(u,\rho,m,\mu)$ of \eqref{Eq: Ergodic MFG}.
\end{theorem}

\begin{proof}
    By Lemma \ref{Lem: Stochastic Interpretation}, choosing $a(\cdot) \coloneqq -D_pH(D_xu(\cdot),\mu)$ gives
    $$
    \begin{aligned}
        \int_{\Omega \times \R^n} L(\alpha,\mu) d\mu(x,\alpha) &= \lim_{T \to \infty} \frac{1}{T} E \int_0^T (L(a(X_t),\mu) + F(\mu,X_t)) dt - \int_\Omega F(\mu,x) dm \\
        &\leq 2C_0 + \lim_{T \to \infty} \frac{1}{T} E \int_0^T L(b_0(\widetilde{X}_t),\mu) dt \\
        &= 2C_0 +  \lim_{T \to \infty} \frac{1}{T} \int_0^T \int_{\Omega \times \R^n} L(\alpha,\mu) d\widetilde{\mu}dt
    \end{aligned}
    $$
    where $b_0(\cdot)$ is given by \eqref{Eq: b0}, $d\widetilde{X}_t = b_0(\widetilde{X}_t)dt + \sqrt{2\sigma}dB_t$, and $\widetilde{\mu} = \mathcal{L}(\widetilde{X}_t,b_0(\widetilde{X}_t))$. By A\ref{A: LL monotonicity}, we get
    $$
    \begin{aligned}
        &\int_{\Omega \times\R^n} L(\alpha,\widetilde{\mu}) d\mu(x,\alpha) + \int_{\Omega \times\R^n} L(\alpha,\mu) d\widetilde{\mu}(x,\alpha) \\
        &\qquad \leq \int_{\Omega \times\R^n} L(\alpha,\mu) d\mu(x,\alpha) + \int_{\Omega \times\R^n} L(\alpha,\widetilde{\mu}) d\widetilde{\mu}(x,\alpha)
    \end{aligned}
    $$
    for all $t$. By A\ref{A: Bound for L},
    $$\int_{\Omega \times\R^n} L(\alpha,\widetilde{\mu}) d\widetilde{\mu}(x,\alpha) = \int_\Omega L(b_0,\widetilde{\mu}) d\widetilde{m} \leq C_0\left(1 + \Lambda_{q'}(\widetilde{\mu})^{q'} + \int_\Omega |b_0|^{q'} d\widetilde{m}\right)$$
    where $\widetilde{m} = \mathcal{L}(\widetilde{X}_t)$. Furthermore, by A\ref{A: Lower bound for L}, we have
    $$|\alpha|^{q'} \leq C_0L(\alpha,\widetilde{\mu}) + C_0^2(1 + \Lambda_{q'}(\widetilde{\mu})^{q'}).$$
    Hence,
    $$
    \begin{aligned}
        \Lambda_{q'}(\mu)^{q'} &= \int_{\Omega \times \R^n} |\alpha|^{q'} d\mu(x,\alpha) \\
        &= \lim_{T \to \infty} \frac{1}{T} \int_0^T \int_{\Omega \times \R^n} |\alpha|^{q'} d\mu(x,\alpha)dt \\
        &\leq C_0^2\left(1 + \lim_{T \to \infty} \frac{1}{T} \int_0^T \Lambda_{q'}(\widetilde{\mu})^{q'}dt\right) + C_0\lim_{T \to \infty} \frac{1}{T} \int_0^T \int_{\Omega \times \R^n} L(\alpha,\widetilde{\mu}) d\mu(x,\alpha) dt \\
        &\leq C_0^2\left(1 + \lim_{T \to \infty} \frac{1}{T} \int_0^T \Lambda_{q'}(\widetilde{\mu})^{q'}dt\right) \\
        &\qquad+ C_0\lim_{T \to \infty} \frac{1}{T} \left[\int_0^T \int_{\Omega \times \R^n} L(\alpha,\mu) d(\mu - \widetilde{\mu})(x,\alpha)dt + \int_0^T\int_{\Omega \times \R^n} L(\alpha,\widetilde{\mu}) d\widetilde{\mu}(x,\alpha)dt\right] \\
        &\leq C_0^2\left(4 + 2\lim_{T \to \infty} \frac{1}{T} \int_0^T \Lambda_{q'}(\widetilde{\mu})^{q'}dt + \lim_{T \to \infty} \frac{1}{T} \int_0^T \int_\Omega |b_0|^{q'} d\widetilde{m}dt\right) \\
        &= 4C_0^2 + 3C_0^2\lim_{T \to \infty} \frac{1}{T} E\int_0^T |b_0|^{q'} dt.
    \end{aligned}
    $$
    By a slight modification to \cite[Theorem VII.3]{lasry1989nonlinear}, we have
    $$\lim_{T \to \infty} \frac{1}{T} E\int_0^T |b_0|^{q'} dt = C_q\widetilde{\rho},$$
    thus completing the proof.
\end{proof}

\section{Existence and Uniqueness Results}
\label{Sec: MFG}

\subsection{Existence of Solutions}
\label{Sec: Existence}

\begin{theorem} \label{Thm: Existence}
    Assume A\ref{A: F}-\ref{A: Regularity of G} and either A\ref{A: Asymptotics of DpH} or A\ref{A: q=2} hold. Then there exists a solution to \eqref{Eq: Ergodic MFG}.
\end{theorem}

\begin{proof}
    Fix $x_0 \in \Omega$. Now let $X$ denote the set of $\nu \in \mathcal{P}_{q'}(\overline{\Omega} \times \R^n)$ such that $\Lambda_{q'}(\nu) \leq C$, where $C \geq 1$ is a constant such that $\Lambda_{q'}(\mu) \leq C$ for every solution of \eqref{Eq: Ergodic MFG} (see Section \ref{Sec: A Priori Estimates}). Given $\widetilde{\mu} \in X$, define $(u,\rho,m)$ to be the unique solution to
    \begin{equation}
        \begin{cases} \label{Eq: MFG for Fixed-Point Map}
            -\sigma\Delta u + H(D_xu,\widetilde{\mu}) + \rho = F(\widetilde{\mu},x) \\
            \sigma\Delta m + \nabla \cdot (mD_pH(D_xu,\widetilde{\mu})) = 0 \\
            m \geq 0, \qquad \int_\Omega m dx = 1, \qquad \underset{d(x) \to 0}{\lim} u(x) = \infty
        \end{cases}
    \end{equation}
    with $u(x_0) = 0$. By Sobolev embedding, we get $D_xu \in C^0(\Omega;\R^n)$. Thus, using Lemma \ref{Lem: Existence & Bound for mu}, we can define $T: X \to X$ as follows. If $\mu = (I,-D_pH(D_xu,\mu)) \# m \in X$, define $T(\widetilde{\mu}) = \mu$; otherwise, define $T(\widetilde{\mu}) = (I,\frac{C}{\Lambda_{q'}(\mu)}I)\sharp \mu$. 

    We observe that $X$ is a compact subset of $\mathcal{P}(\overline{\Omega} \times \R^n)$ by tightness. To prove that $T$ is continuous, take a sequence $\widetilde{\mu}_k \to \widetilde{\mu}$ with $\widetilde{\mu}_k \in X$, and let $(u_k,\rho_k,m_k)$ be the corresponding solutions of \eqref{Eq: MFG for Fixed-Point Map}. Note that $\rho_k$ is bounded and that $u_k$ is bounded in $W^{2,r}(K)$ for all $r > 0$ and $K \subset\subset \Omega$. Hence, we can choose some $\rho \in \R$ and $u \in W^{1,r}_{loc}(\Omega)$ so that, up to a subsequence, $\rho_k \to \rho$ and $u_k \to u$ in $W^{2,r}(K)$ for all $r > 0$ and $K \subset\subset \Omega$. Hence,
    $$-\sigma\Delta u + H(D_xu,\widetilde{\mu}) + \rho = F(\widetilde{\mu},x).$$
    By uniqueness, we get the convergence of the entire sequence.

    Now note that by Theorem \ref{Thm: Regularity of m}, $m_k$ is a bounded sequence in $H^1(\Omega) \cap L^\infty(\Omega)$. Hence, there is some $m \in L^\infty(\Omega) \cap H^1(\Omega)$ such that, passing to a subsequence if necessary, $m_k \to m$ a.e., strongly in $L^r(\Omega)$ for $r \geq 1$, and weakly in $H^1(\Omega)$. To show that the entire sequence converges, we will show that $m$ satisfies \eqref{Eq: FPK} for $b = -D_pH(D_xu,\widetilde{\mu})$, and the result will follow by uniqueness. To this end, let $\phi \in L^\infty(\Omega)$ be such that $-\sigma\Delta\phi + D_pH(D_xu,\widetilde{\mu}) \cdot D_x\phi = \eta$ in the sense of distributions for some $\eta \in W^{1,\infty}(\Omega)$. By \cite[Proposition 3.9]{porretta2023ergodic}, we can let $\phi_k$ be the solution of
    $$-\sigma\Delta\phi_k + D_pH(D_xu_k,\widetilde{\mu}_k) \cdot D_x\phi_k = \eta + \lambda_k.$$
    Then we have
    $$\int_\Omega(-\sigma\Delta\phi + D_pH(D_xu,\widetilde{\mu}) \cdot \phi + \lambda_k)dm_k = \int_\Omega(-\sigma\Delta\phi_k + D_pH(D_xu_k,\widetilde{\mu}_k) \cdot \phi_k)dm_k = 0.$$
    Since $|\lambda_k| \leq \|\eta\|_\infty \leq C$ by \cite[Equation (3.33)]{porretta2023ergodic}, there is some $\lambda \in \R$ such that, taking a subsequence if necessary, $\lambda_k \to \lambda$ and hence
    $$\int_\Omega (-\sigma\Delta\phi + D_pH(D_xu,\widetilde{\mu}) \cdot D_x\phi + \lambda)dm = 0.$$
    On the other hand, note that for all $\xi \in C^\infty_c(\Omega)$, we have
    $$\int_\Omega D_x\phi_k \cdot (\sigma D_x\xi + D_pH(D_xu_k,\widetilde{\mu}_k)\xi) dx =  \int_\Omega (\eta + \lambda_k)\xi dx.$$
    By \cite[Proposition 3.5]{porretta2023ergodic}, we get bounds for $\phi_k$ in $W^{1,r}(\Omega)$ for all $r \geq 1$, uniformly in $k$. Thus, passing to a subsequence if necessary, we can find some $\widetilde{\phi} \in H^1(\Omega)$ with $D_x\phi_k \to D_x\widetilde{\phi}$ weakly in $L^2$. Since $D_pH(D_xu_k,\widetilde{\mu}_k)$ is locally uniformly bounded and $D_pH(D_xu_k,\widetilde{\mu}_k) \to D_pH(D_xu,\widetilde{\mu})$ a.e., passing to the limit gives
    $$\int_\Omega D_x\widetilde{\phi} \cdot (\sigma D_x\xi + D_pH(D_xu,\widetilde{\mu})) dx = \int_\Omega(\eta + \lambda)\xi dx.$$
    In particular, $\phi,\widetilde{\phi}$ are weak solutions to
    $$\sigma\Delta\phi + D_pH(D_xu,\widetilde{\mu}) \cdot D_x\phi = \eta, \qquad \sigma\Delta\widetilde{\phi} + D_pH(D_xu,\widetilde{\mu}) \cdot D_x\widetilde{\phi} = \eta + \lambda.$$
    Thus, \cite[Proposition 3.9]{porretta2023ergodic} gives us that $\lambda = 0$.

    Let $\widetilde{m}$ be the unique solution to
    $$\sigma\Delta\widetilde{m} + \nabla \cdot (\widetilde{m}D_pH(D_xu,\widetilde{\mu})) = 0$$
    and let $\xi \in W^{1,\infty}(\Omega)$. Again, by \cite[Proposition 3.9]{porretta2023ergodic}, there is some $\phi \in W^{1,r}(\Omega)$ with $-\sigma\Delta\phi + D_pH(D_xu,\mu) \cdot D_x\phi = \xi + \lambda_\xi$ for some $\lambda_\xi \in \R$. Finally, using $\phi$ as a test function gives us that
    $$\int_\Omega (\xi + \lambda_\xi)dm = \int_\Omega (\xi + \lambda_\xi)d\widetilde{m} = 0$$
    and hence
    $$\int_\Omega \xi dm = \int_\Omega \xi d\widetilde{m}.$$
    Therefore, we conclude $m = \widetilde{m}$. Finally, we apply Lemma \ref{Lem: Continuity of mu} to get that $T(\widetilde{\mu}_k) \to T(\widetilde{\mu})$. By the Schauder fixed-point theorem, it follows that $T$ has a fixed-point $\mu$ and hence \eqref{Eq: Ergodic MFG} has a solution.
\end{proof}

\subsection{Uniqueness of Solutions}
\label{Sec: Uniqueness}

We conclude by proving the uniqueness of solutions to our system. The argument is adapted from \cite{kobeissi2022mean} and uses Lasry-Lions monotonicity to obtain uniqueness.
\begin{theorem} \label{Thm: Uniqueness}
    Assume A\ref{A: F}-\ref{A: Regularity of G} and either A\ref{A: Asymptotics of DpH} or A\ref{A: q=2} hold. Then there is at most one solution $(u,\rho,m,\mu)$ to \eqref{Eq: Ergodic MFG}, up to adding a constant to $u$.
\end{theorem}

\begin{proof}
    Let $(u_1,\rho_1,m_1,\mu_1)$ and $(u_2,\rho_2,m_2,\mu_2)$ be solutions.
    Then by lemma \ref{Lem: u as a Test Function} and A\ref{A: F}, we get
    $$
    \begin{aligned}
        0 \leq& \int_\Omega m_1(H(D_xu_1,\mu_1) - H(D_xu_2,\mu_2) + D_x(u_2 - u_1) \cdot D_pH(D_xu_1,\mu_1)) dx \\
        &+ \int_\Omega m_2(H(D_xu_2,\mu_2) - H(D_xu_1,\mu_1) + D_x(u_1 - u_2) \cdot D_pH(D_xu_2,\mu_2)) dx.
    \end{aligned}
    $$
    Note that for $i = 1,2$, letting $\alpha^{\mu_i} \coloneqq -D_pH(D_xu_i,\mu_i)$, we get
    $$L(\alpha^{\mu_i},\mu_i) = D_xu_i \cdot D_pH(D_xu_i,\mu_i) - H(D_xu_i,\mu_i)$$
    and
    $$D_xu_i = -D_\alpha L(\alpha^{\mu_i},\mu_i).$$
    Thus,
    $$
    \begin{aligned}
        0 =& \int_\Omega m_1(L(\alpha^{\mu_2},\mu_2) - L(\alpha^{\mu_1},\mu_1) + D_\alpha L(\alpha^{\mu_2},\mu_2) \cdot (\alpha^{\mu_1}-\alpha^{\mu_2})) dx \\
        &+ \int_\Omega m_2(L(\alpha^{\mu_1},\mu_1) - L(\alpha^{\mu_2},\mu_2) + D_\alpha L(\alpha^{\mu_1},\mu_1) \cdot (\alpha^{\mu_2}-\alpha^{\mu_1})) dx.
    \end{aligned}
    $$
    Since $L$ is strictly convex,
    \begin{equation}
        L(\alpha_1,\mu) - L(\alpha_2,\mu) + D_\alpha L(\alpha_1,\mu) \cdot (\alpha_2 - \alpha_1) \leq 0
        \label{Eq: Convexity of L}
    \end{equation}
    with equality holding if and only if $\alpha_1 = \alpha_2$. Hence,
    $$
    \begin{aligned}
        0 \leq& \int_\Omega \bigg(m_1(L(\alpha^{\mu_1},\mu_2) - L(\alpha^{\mu_1},\mu_1)) + m_2(L(\alpha^{\mu_2},\mu_1) - L(\alpha^{\mu_2},\mu_2))\bigg)dx \\
        =& -\int_{\Omega \times \R^n} (L(\alpha,\mu_1) - L(\alpha,\mu_2)) d(\mu_1 - \mu_2)(x,\alpha).
    \end{aligned}
    $$
    By A\ref{A: LL monotonicity}, this gives
    $$\int_\Omega \bigg(m_1(L(\alpha^{\mu_2},\mu_1) - L(\alpha^{\mu_1},\mu_1)) + m_2(L(\alpha^{\mu_1},\mu_2) - L(\alpha^{\mu_2},\mu_2))\bigg)dxdt = 0.$$
    By the condition for equality for \eqref{Eq: Convexity of L}, we get $|\{x \in \Omega : \alpha^{\mu_1} \neq \alpha^{\mu_2}, m_i \neq 0\}| = 0$ for $i = 1,2$. Therefore, $\alpha^{\mu_1} = \alpha^{\mu_2}$. By the uniqueness of solutions to the Fokker-Planck equation, $m_1 = m_2$. Therefore, $\mu_1 = (I,\alpha^{\mu_i}) \# m_i = \mu_2$. By the uniqueness of solutions to the Hamilton-Jacobi equation, $u_1 = u_2 + C$ and $\rho_1 = \rho_2$.
\end{proof}

\section{Acknowledgments}

We would like to thank Alessio Porretta for his technical assistance regarding the Fokker-Planck equation. Additionally, we are grateful to be supported by National Science Foundation through NSF Grant DMS-2045027.

\bibliographystyle{amsplain}
\bibliography{refs}

\end{document}